\documentclass[10pt,a4paper]{article}
\usepackage{amsmath}
\usepackage{a4wide}
\usepackage{mathrsfs}
\usepackage{stmaryrd}
\usepackage{pifont}
\usepackage{amssymb}
\usepackage{color}
\usepackage{amsthm}
\usepackage{graphicx}
\usepackage[colorlinks=true]{hyperref}
\newcommand{\udots}{\mathinner{\mskip1mu\raise1pt\vbox{\kern7pt\hbox{.}}
\mskip2mu\raise4pt\hbox{.}\mskip2mu\raise7pt\hbox{.}\mskip1mu}}

\hypersetup{linkcolor=blue, urlcolor=red, citecolor=red}
\allowdisplaybreaks[4] \numberwithin{equation}{section}
\newtheorem{theorem}{Theorem}[section]

\newtheorem{lemma}{Lemma}[section]
\newtheorem{proposition}{Proposition}[section]

 \input amssym.def
\input amssym
\title{Stabilization on periodic impulse control systems
\thanks{The authors acknowledge the financial support by
the NNSF of China under grants 11571264, 11601377.}
}

\author{Shulin Qin\thanks{School of Science, Tianjin University of Commerce, Tianjin 300134, China (\texttt{shulinqin@yeah.net})},\quad Gengsheng Wang\thanks{Center for Applied Mathematics, Tianjin University, Tianjin 300072, China (\texttt{wanggs62@yeah.net})},\quad Huaiqiang Yu\thanks{School of Mathematics, Tianjin University, Tianjin 300354, China (\texttt{huaiqiangyu@tju.edu.cn})}}
\date{}
\begin{document}
\maketitle
\begin{abstract}
This paper studies the stabilization for a kind of linear and impulse control systems in finite-dimensional spaces, where impulse instants appear  periodically. We present several characterizations on the stabilization; show how to  design feedback laws;
and provide locations for impulse  instants to ensure the stabilization. In the proofs of these results, we set up  a discrete LQ problem; derived a discrete dynamic programming principle, built up a variant of Riccati's equation; applied repeatedly the  Kalman
controllability decomposition; and used a controllability result built up in \cite{b2}.

\end{abstract}
\vskip 10pt

\section{Introduction}\label{sec_intro}\label{introductionsection}

\subsection{Control system and strategy}\label{subsection1.1}

Given a state matrix $A\in \mathbb{R}^{n\times n}$, a number $\hbar\in \mathbb{N}^+:=\{1,2,\ldots\}$, $\hbar$
 control matrices $\{B_k\}_{k=1}^{\hbar}\subset \mathbb{R}^{n\times m}$ and
impulse instants $\Lambda_{\hbar}:=\{t_j\}_{j\in \mathbb{N}}$ (Here, $\mathbb{N}:=\{0,1,2,\ldots\}$.) with
\begin{eqnarray}\label{timeinstants}
t_0:=0<t_1<t_2<\cdots\;\;\mbox{and}\;\; t_{j+\hbar}-t_j=t_{\hbar}\;\forall\; j\in \mathbb{N}^+,
\end{eqnarray}
we consider the impulse control system:
\begin{equation}\label{yu-10-24-1}
\begin{cases}
    x'(t)=Ax(t),&t\in\mathbb{R}^+\setminus\Lambda_{\hbar},\\
    \triangle x(t_{j})=B_{\vartheta(j)}u_{j},&j\in\mathbb{N}^+,
\end{cases}
\end{equation}
where $\mathbb{R}^+:=[0,\infty)$,
$\triangle x(t_j):=x(t_j^+)-x(t_j)$,
    $u:=(u_j)_{j\in\mathbb{N}^+}\in l^2(\mathbb{N}^+;\mathbb{R}^m)$
and
\begin{eqnarray}\label{index1}
\vartheta(j):=j-\left[{j}/{\hbar}\right]\hbar\;\; \forall\; j\in \mathbb{N}^+.
\end{eqnarray}
Here, $[s]:=\max\{k\in\mathbb{N}\;: \; k<s\}$ for each $s>0$.
(Notice that for each $1\leq j\leq \hbar$ and each $k\in \mathbb{N}^+$, we have
$\vartheta(j+k\hbar)=\vartheta(j)=j$.)
Several notes are given in order:
\begin{itemize}
\item Throughout the paper, $\hbar\in \mathbb{N}^+$ is arbitrarily fixed and
  $\Lambda_{\hbar}$ denotes
  an element in the set:
  \begin{eqnarray}\label{instantset}
    \mathfrak{I}_\hbar:=\left\{\Lambda_{\hbar}=\{t_j\}_{j\in \mathbb{N}}\;   : \;t_{j+1}>t_j>t_0=0\;\;\mbox{and}\;\;
    t_{j+\hbar}-t_j=t_{\hbar}\; \forall  j\in\mathbb{N}^+\right\}.
    \end{eqnarray}
     Each    $\Lambda_{\hbar}=\{t_j\}_{j\in \mathbb{N}}\in\mathfrak{I}_{\hbar}$ satisfies
\begin{eqnarray}\label{1.5201967}
    \{t_{j\hbar+k}-t_{j\hbar}\}_{k=1}^\hbar=\{t_{j\hbar+k}-jt_{\hbar}\}_{k=1}^\hbar
    =\{t_k\}_{k=1}^\hbar
    \;\;\forall j\in\mathbb{N}.
\end{eqnarray}
         Because of such periodicity, we call (\ref{yu-10-24-1}) an \emph{$\hbar$-periodic impulse control system}.

\item The  control strategy in \eqref{yu-10-24-1}
can be explained by two ways.
Way One: With  $\hbar$ control matrices  $\{B_j\}_{j=1}^\hbar$ and
  impulse instants $\{t_j\}_{j\in \mathbb{N}}$ (obeying \eqref{timeinstants}) in hands, we put periodically the control matrices
into the system $x'=Ax$ at the impulse instants. Way Two: With $\hbar$ control matrices  $\{B_k\}_{k=1}^\hbar$ in hands, we first  choose impulse instants $\{t_j\}_{j\in \mathbb{N}}$ satisfying  \eqref{timeinstants}, then put the control matrices  periodically into the system $x'=Ax$ at the impulse instants.
It deserves mentioning that \eqref{yu-10-24-1} contains only $\hbar$ control matrices, but infinitely many controls.

 In the first way mentioned above, we  denote the system
    (\ref{yu-10-24-1}) by $(A, \{B_k\}_{k=1}^{\hbar},\Lambda_{\hbar})$. In the second way, we treat the system
    (\ref{yu-10-24-1})  as a pair $(A, \{B_k\}_{k=1}^{\hbar})$, while
    treat $\Lambda_{\hbar}$  as an auxiliary
     of controls $(u_j)_{j\in \mathbb{N}^+}$.

\item The control system  \eqref{yu-10-24-1} can be understood as
a model describing a kind of  multi-person cooperation.

 \item When $\hbar=1$, we necessarily have a constant $\tau>0$ so that $t_j=j\tau$ for all $j\in \mathbb{N}$. However, the case that $B=B_k\;\forall\; k$ may correspond to any
     $\hbar\in \mathbb{N}^+$ and
       any  $\{t_j\}_{j\in \mathbb{N}}$ satisfying \eqref{timeinstants}.
    {\it  When $B_k=B\in \mathbb{R}^{n\times m}\; \forall\; k\in \{1,2,\ldots,\hbar\}$, we simply write
     $\{B\}$ for
$\{B_k\}_{k=1}^\hbar$ if there is no risk causing any confusion.}

    \item One can easily check that for each $x_0\in \mathbb{R}^n$ and each  $u:=(u_j)_{j\in\mathbb{N}^+}\in
        l^2(\mathbb{R}^m)$,
     the system (\ref{yu-10-24-1}), with the initial condition  $x(0)=x_0$, has a unique solution
     $x(\cdot;u,x_0)$ in $\mathcal{P}C(\mathbb{R}^+;\mathbb{R}^n)$,
the space of all functions from $\mathbb{R}^+$ to $\mathbb{R}^n$, which are left continuous
over $\mathbb{R}^+$,
 continuous over
     $\mathbb{R}^+\setminus\{t_j\}_{j\in\mathbb{N}^+}$, and have discontinuities of first kind at the points
     $\{t_j\}_{j\in\mathbb{N}^+}$. ({\it Here and throughout the paper, $l^2(\mathbb{R}^d)$, $d\in \mathbb{N}^+$, stands for
     $l^2(\mathbb{N}^+;\mathbb{R}^d)$. The same is said about $l^\infty(\mathbb{R}^d)$.}) Furthermore, we have
 \begin{eqnarray*}\label{solutionexpress1}
    x(t;u,x_0)=e^{At}x_0+\sum_{0<t_{j}<t}e^{A(t-t_{j})}B_{\vartheta(j)}u_{j}\;\;\mbox{for any}\;t\in\mathbb{R}^+.
\end{eqnarray*}

\item The way that $\{B_k\}_{k=1}^\hbar$ and $(u_j)_{j\in \mathbb{N}^+}$ affect the system \eqref{yu-10-24-1}
differs from the way that $B$ and $v$ affect the usual control system:
\begin{equation}\label{421usualcontrolsystem}
x'(t)=Ax(t)+Bv(t),\;\;t\in\mathbb{R}^+,\;\;\mbox{where}\;\;v\in L^2(\mathbb{R}^+; \mathbb{R}^{m}).
\end{equation}
Notice that the system \eqref{421usualcontrolsystem} is time-invariant, while the system \eqref{yu-10-24-1}
is time-varying in the sense: control matrices varies at impulse instants $\hbar$-periodically.
\end{itemize}

\subsection{Main problems}\label{subsection1.2}
 We begin with introducing several concepts.

\begin{itemize}

\item
The system $(A,  \{B_k\}_{k=1}^{\hbar},\Lambda_{\hbar})$ (or \eqref{yu-10-24-1})  is
said to be
$\hbar$-stabilizable if  there is a sequence of
feedback laws $\{F_k\}_{k=1}^{\hbar}\subset\mathbb{R}^{m\times
n}$
so that
 the following closed-loop system is stable:
\begin{equation}\label{yu-10-24-4}
\begin{cases}
 x'(t)=Ax(t),&t\in\mathbb{R}^+\setminus\Lambda_{\hbar},\\
    \triangle x(t_{j})=B_{\vartheta(j)}
    F_{\vartheta(j)}x(t_{j}),&j\in\mathbb{N}^+.
\end{cases}
\end{equation}
Here, the stability of \eqref{yu-10-24-4} means that  there is $M>0$ and $\mu>0$ so that any solution
$x_\mathcal{F}(\cdot)$ to \eqref{yu-10-24-4} satisfies
\begin{eqnarray}\label{stbleexpression}
\|x_\mathcal{F}(t)\|_{\mathbb{R}^n}\leq Me^{-\mu t}\|x_\mathcal{F}(0)\|_{\mathbb{R}^n}\;\;\forall\; t\in\mathbb{R}^+.
\end{eqnarray}
{\it We simply write $\mathcal{F}:=\{F_k\}_{k=1}^{\hbar}$ (call it a feedback law) and  denote the
closed-loop system  \eqref{yu-10-24-4} by  $(A, \{B_kF_k\}_{k=1}^\hbar, \Lambda_\hbar)$.}
 Since  $\{F_k\}_{k=1}^{\hbar}$ appear at time instants
$\Lambda_{\hbar}$ $\hbar$-periodically, {\it the feedback law $\mathcal{F}$ is indeed $\hbar$-periodic time-varying.}

\item A pair $(A, \{B_k\}_{k=1}^{\hbar})$ is  said to be
$\hbar$-stabilizable if
 there is $\Lambda_\hbar$ so that the system $(A, \{B_k\}_{k=1}^{\hbar}, \Lambda_\hbar)$ is
 $\hbar$-stabilizable.

\end{itemize}

This  paper mainly concerns the following  problems on  the stabilization for  the system \eqref{yu-10-24-1}:

\begin{itemize}
\item What is the characterization of  the $\hbar$-stabilization for a system $(A,
    \{B_k\}_{k=1}^{\hbar},\Lambda_{\hbar})$?

    \item When a system $(A,  \{B_k\}_{k=1}^{\hbar},\Lambda_{\hbar})$ is $\hbar$-stabilizable, how to design a
        feedback law?

        \item  What is the characterization of  the $\hbar$-stabilization  for a pair
        $(A, \{B_k\}_{k=1}^{\hbar})$?

        \item When   a pair $(A, \{B_k\}_{k=1}^{\hbar})$ is $\hbar$-stabilizable, how to choose
        $\Lambda_{\hbar}$  so that $(A, \{B_k\}_{k=1}^{\hbar}, \Lambda_{\hbar})$ is $\hbar$-stabilizable?
\end{itemize}

We now explain why  these problems deserve to be studied. First, in the classical control theory of linear ODEs,  the
characterization on the stabilization for the control system \eqref{421usualcontrolsystem} (or $(A,B)$)
   is the well-known Kalman's
criterion: $\mbox{Rank}\;(\lambda I-A,B)=n$ for all $\lambda\in \mathbb{C}^+:=\{z\in\mathbb{C}:\mbox{Re}\;z\geq 0\}$.
When $(A,B)$ is stabilizable, the feedback law can be obtained from the Riccati equation. These constitute fundamental stabilization theory for the control system \eqref{421usualcontrolsystem}. From this point of view,  the first three problems
mentioned-above are fundamental on  the stabilization for the periodic impulse control system \eqref{yu-10-24-1}.
Second, $\Lambda_{\hbar}$ gives locations where control matrices are put and controls are active.
This shows the importance of
 the last problem mentioned-above.

\subsection{Main results}\label{subsection1.3}

The first main theorem concerns characterizations of $\hbar$-stabilization for a system $(A,
\{B_k\}_{k=1}^{\hbar}, \Lambda_\hbar)$ and the design of a feedback law. We start with the following
notations:
\begin{eqnarray*}\label{notation1.9}
\mathfrak{M}^{d}_{\hbar}:=\{(M_j)_{j\in\mathbb{N}^+}\in l^\infty(\mathbb{R}^{d\times d})\;:
\;M_{j+\hbar}=M_j\;\forall j\in\mathbb{N}^+\};
\end{eqnarray*}
\begin{eqnarray*}
\mathfrak{M}^{d}_{\hbar,+}:=\{(M_j)_{j\in\mathbb{N}^+}\in \mathfrak{M}_{\hbar}^{d}\;:
\; \mbox{each}\; M_j\; \mbox{is symmetric and positive definite} \},\; d\in \mathbb{N}^+.
\end{eqnarray*}
 Arbitrarily fix $\mathcal{Q}:=(Q_j)_{j\in\mathbb{N}^+}\in\mathfrak{M}^{n}_{\hbar,+}$ and
 $\mathcal{R}:=(R_j)_{j\in\mathbb{N}^+}\in\mathfrak{M}^{m}_{\hbar,+}$. We consider the
 LQ  problem (associated with a control system  $(A,\{B_k\}_{k=1}^{\hbar},\Lambda_{\hbar})$,
 where $\Lambda_{\hbar}=\{t_j\}_{j\in\mathbb{N}}\in\mathfrak{I}_{\hbar}$):
  \vskip 5pt
   \textbf{(I-I-LQ)}: \;\;  \emph{Given $x_0\in\mathbb{R}^n$, find a control
   $u^*=(u^*_j)_{j\in\mathbb{N}^+}\in\mathcal{U}_{ad}(x_0)$ so that
\begin{equation*}
    J(u^*;x_0)=\inf_{u\in\mathcal{U}_{ad}(x_0)}J(u;x_0),
\end{equation*}}
 where
 \begin{equation}\label{yu-11-19-2}
    \mathcal{U}_{ad}(x_0):=\{u=(u_j)_{j\in\mathbb{N}^+}\in
    l^2(\mathbb{R}^m):(x(t_j;u,x_0))_{j\in\mathbb{N}^+}\in l^2(\mathbb{R}^n)\};
\end{equation}
 \begin{equation}\label{yu-11-19-1}
    J(u;x_0):=\sum_{j=1}^{+\infty}\left[\langle Q_jx(t_j;u,x_0),x(t_j;u,x_0)\rangle_{\mathbb{R}^n}+\langle
    R_ju_j,u_j\rangle_{\mathbb{R}^m}\right],\; u\in\mathcal{U}_{ad}(x_0).
\end{equation}
Here, $\langle \cdot,\cdot\rangle_{\mathbb{R}^n}$ and $\langle \cdot,\cdot\rangle_{\mathbb{R}^m}$
  stands for the usual inner products in $\mathbb{R}^n$ and $\mathbb{R}^m$. {\it In this paper, we simply denote them
  by $\langle \cdot,\cdot\rangle$ if there is no risk causing any confusion.}

  Next, we introduce the  variant of Riccati's equation (which is associated with
 \textbf{(I-I-LQ)}):
   \begin{equation}\label{yu-11-26-1}
\begin{cases}
    e^{-A^\top(t_{k+1}-t_k)}P_ke^{-A(t_{k+1}-t_k)}-P_{k+1}\\
     \;\;\;\;\;=Q_{k+1}-P_{k+1}B_{k+1}(R_{k+1}+B_{k+1}^\top P_{k+1}B_{k+1})^{-1}B_{k+1}^\top P_{k+1},&0\leq
     k\leq \hbar-1,\\
    P_0=P_{\hbar}.
\end{cases}
\end{equation}
 Several notes on \textbf{(I-I-LQ)} and \eqref{yu-11-26-1} are given in order.
   \begin{itemize}
   \item  Double \textbf{I} in the notation \textbf{(I-I-LQ)} denotes
    the abbreviations of {\it infinite horizon} and {\it impulse controls}.
    In this LQ problem,   $ \mathcal{U}_{ad}(x_0)$ is called an \emph{admissible set}, which is independent of the choice of $\mathcal{Q}$ and $\mathcal{R}$, while $ J(\cdot;x_0)$ is called a \emph{cost functional} which depends on  the choice of $\mathcal{Q}$ and $\mathcal{R}$.
           \item In \eqref{yu-11-26-1}, unknowns $P_k$, $k=0,1,\dots \hbar$, are $n\times n$ real, symmetric and  positive  definite matrices. The solution of \eqref{yu-11-26-1}, if  exists, is denoted by $\{P_k\}_{k=0}^\hbar$.
  \end{itemize}

\begin{theorem}\label{yu-theorem-11-27-1}
Given $(A,\{B_k\}_{k=1}^\hbar,\Lambda_{\hbar})$,
     the following  statements are equivalent:
\begin{enumerate}
  \item[(i)] The system $(A,\{B_k\}_{k=1}^\hbar,\Lambda_{\hbar})$ is  $\hbar$-stabilizable.
  \item[(ii)] For each  $x_0\in\mathbb{R}^n$,   the admissible set $\mathcal{U}_{ad}(x_0)$ is not empty.
\item[(iii)] For any $\mathcal{Q}\in \mathfrak{M}^{n}_{\hbar,+}$
and $\mathcal{R}\in \mathfrak{M}^{m}_{\hbar,+}$, the equation (\ref{yu-11-26-1}) has a unique solution
$\{P_{k}\}_{k=0}^{\hbar}$.
  \item[(iv)] There is $\mathcal{Q}\in \mathfrak{M}^{n}_{\hbar,+}$
and $\mathcal{R}\in \mathfrak{M}^{m}_{\hbar,+}$ so that the equation (\ref{yu-11-26-1}) has a unique solution
$\{P_{k}\}_{k=0}^{\hbar}$.
\end{enumerate}
    Furthermore, if one of above items is true, then the feedback
    law   $\mathcal{F}=\{F_k\}_{k=1}^{\hbar}$ can be designed in the following manner:
    First, take arbitrarily $\mathcal{Q}\in \mathfrak{M}^{n}_{\hbar,+}$
and $\mathcal{R}\in \mathfrak{M}^{m}_{\hbar,+}$, then solve (\ref{yu-11-26-1}) to get   $\{P_k\}_{k=0}^\hbar$, finally set
\begin{equation}\label{yu-12-4-1-b}
   F_k:=-\left(R_{k}+B_{k}^\top P_{k}B_{k}\right)^{-1}B_{k}^\top
    P_{k}\;\;\mbox{for each}\; k=1,\dots,\hbar.
\end{equation}
   \end{theorem}

Several notes on Theorem \ref{yu-theorem-11-27-1} are given in order:
\begin{itemize}

\item In Theorem \ref{yu-theorem-11-27-1}, our feedback controls are as: $(u_j)_{j\in \mathbb{N}^+}=(F_{\vartheta(j)}x(t_{j}))_{j\in\mathbb{N}^+}$.
When  we replace $\triangle x(t_j)$ by $\triangle^-x(t_j):=x(t_j)-x(t_j^-)$
in \eqref{yu-10-24-1}, we can get the  same results as those in Theorem \ref{yu-theorem-11-27-1},
but  feedback controls should be
$(u_{j})_{j\in\mathbb{N}^+}=(F_{\vartheta(j)}x(t^-_{j}))_{j\in\mathbb{N}^+}$.

\item If $\hbar=1$, $B_1=B$, $\Lambda_{1}=\{j\tau\}_{j\in \mathbb{N}}$ (with $\tau>0$), $Q_1=\mathbb{I}_n$ and $R_1=\mathbb{I}_m$, then  \eqref{yu-11-26-1} reads:
    \begin{equation}\label{yu-4-19-1}
    e^{-A^\top\tau}Pe^{-A\tau}-P=\mathbb{I}_n-PB(\mathbb{I}_m+B^\top PB)^{-1}B^\top P.
\end{equation}
When $(A,\{B\},\Lambda_{1})$ is $1$-stabilizable, the feedback law can be taken as:
\begin{equation*}
   \mathcal{F}=\{F_k\}_{k=1}^{1},\;\mbox{with}\; F_1:=-(\mathbb{I}_m+B^\top PB)^{-1}B^\top P,
\end{equation*}
    where $P$ is the solution of the equation (\ref{yu-4-19-1}).
\end{itemize}

The second main theorem concerns   characterizations of the $\hbar$-stabilization  for a pair $(A,\{B_k\}_{k=1}^{\hbar})$.
\begin{theorem}\label{yu-theorem-3-14-1}
    Given $(A,\{B_k\}_{k=1}^{\hbar})$, the following statements are equivalent:
\begin{enumerate}
  \item[(i)] The pair $(A,\{B_k\}_{k=1}^{\hbar})$ is $\hbar$-stabilizable.
  \item[(ii)] For any $\lambda\in\mathbb{C}^+$, it holds that
      $\mbox{Rank}\;(\lambda\mathbb{I}_n-A,B_1,\cdots,B_\hbar)=n$.
  \item[(iii)] For any $\lambda\in\sigma(A)\cap\mathbb{C}^+$, it holds that
      $\mbox{Rank}\;(\lambda\mathbb{I}_n-A,B_1,\cdots,B_\hbar)=n$.
\end{enumerate}
    Here,
 $\sigma(A)$ denotes the spectrum of $A$.
\end{theorem}
We now give a remark  on Theorem \ref{yu-theorem-3-14-1}.

\begin{itemize}
\item
 By the classical stabilization theory on \eqref{421usualcontrolsystem} and by using  Theorem
 \ref{yu-theorem-3-14-1}, we can easily see that the system  \eqref{421usualcontrolsystem} is stabilizable if and only if
the system \eqref{yu-10-24-1}, where $B_k=B$ for all $k$, is $\hbar$-stabilizable.
This  gives connection between the usual control system \eqref{421usualcontrolsystem} and the periodic impulse control
system \eqref{yu-10-24-1}, from perspective of the stabilization.

\end{itemize}

   The third main theorem gives, for an  $\hbar$-stabilizable pair $(A, \{B_k\}_{k=1}^{\hbar})$, a set of such $\Lambda_{\hbar}$ making $(A, \{B_k\}_{k=1}^{\hbar},
   \Lambda_{\hbar})$  $\hbar$-stabilizable.
   We start with  some notations. Given $(A,\{B_k\}_{k=1}^{\hbar})$, we write
   \begin{eqnarray}\label{1.162019611}
   \mathcal{B}:=\left(
    \begin{array}{cccc}
    B_1 & B_2 &\cdots & B_\hbar \\
    \end{array}
    \right)(\in\mathbb{R}^{n\times m\hbar})
   \end{eqnarray}
       and let
\begin{equation}\label{yu-3-17-3}
    d_A:=\min \left\{\pi/|\mbox{Im}\lambda|:\lambda\in \sigma(A)\right\}.
\end{equation}
    (Here, we use the convention:$\frac{1}{0}=+\infty$.) Given  $C\in\mathbb{R}^{i\times i}$
    and $D\in\mathbb{R}^{i\times k}$ with $i,k\in\mathbb{N}^+$, we write
\begin{equation}\label{yu-3-17-4}
    q^{i,k}(C,D):=\max\{\mbox{dim}\mathcal{V}^{i}_C(d):d\;\mbox{is a column of}\; D\},
\end{equation}
where $\mathcal{V}_C^{i}(d):=\mbox{span}\{d,Cd,\ldots,C^{i-1}d\}$ (the linear subspace
generated by all column vectors $d$, $Cd\dots$, $C^{i-1}d$). Let
\begin{equation}\label{yu-3-18-1}
    \mathfrak{L}_{A,\mathcal{B},\hbar}:=\{\{t_j\}_{j\in\mathbb{N}}\subset\mathbb{R}^+:
    \mbox{Card}((s,s+d_A)\cap \{t_j\}_{j\in\mathbb{N}})\geq \hbar(q^{n,m\hbar}(A,\mathcal{B})+2)\;\forall
    s\in\mathbb{R}^+\},
\end{equation}
where $\mbox{Card}((s,s+d_A)\cap \{t_j\}_{j\in\mathbb{N}})$ denotes the number of elements of the set $(s,s+d_A)\cap
\{t_j\}_{j\in\mathbb{N}}$.

\begin{theorem}\label{yu-theorem-5-13-1}
     If a pair $(A,\{B_k\}_{k=1}^{\hbar})$ is $\hbar$-stabilizable,
     then for any
     $\Lambda_{\hbar}$ in
    $\mathfrak{I}_{\hbar}\cap \mathfrak{L}_{A,\mathcal{B},\hbar}$, the system $(A,\{B_k\}_{k=1}^{\hbar},\Lambda_{\hbar})$ is $\hbar$-stabilizable.
     \end{theorem}
     Several notes on  Theorem \ref{yu-theorem-5-13-1} are given in order:

\begin{itemize}
\item
We have that
$\mathfrak{I}_{\hbar}\cap \mathfrak{L}_{A,\mathcal{B},\hbar}\neq \emptyset$ for any $\hbar\in\mathbb{N}^+$.
Indeed, when $0<\tau<\frac{d_A}{\hbar[q^{n,m\hbar}(A,\mathcal{B})+2]+2}$, we have that $\Lambda_{\hbar}=\{j\tau\}_{j\in\mathbb{N}}\in \mathfrak{I}_{\hbar}\cap
\mathfrak{L}_{A,\mathcal{B},\hbar}$.

\item In the case that  $\sigma(A)\subset\mathbb{R}$, we have $d_A=+\infty$
which implies any infinite sequence $\{t_j\}_{j\in\mathbb{N}}$ of $\mathbb{R}^+$
 with $t_j\to +\infty$ as $j\to +\infty$ is in $\mathfrak{L}_{A,\mathcal{B},\hbar}$.
So if $(A,\{B_k\}_{k=1}^{\hbar})$ is
$\hbar$-stabilizable and $\sigma(A)\subset\mathbb{R}$, then $\forall\; \Lambda_{\hbar}\in \mathfrak{I}_{\hbar}$,
   $(A,\{B_k\}_{k=1}^{\hbar}, \Lambda_{\hbar})$ is
  $\hbar$-stabilizable.

\end{itemize}

\subsection{Novelties of this paper}

\begin{itemize}
\item The control strategy presented in \eqref{yu-10-24-1} seems to be new for us.

 \item It seems for us that  characterizations on the stabilization for impulse control systems
have not been touched upon. (At least, we do not find any such literature.) From this perspective, the
equivalent results in Theorem \ref{yu-theorem-11-27-1}, as well as in Theorem \ref{yu-theorem-3-14-1}, are new.

\item It seems for us that studies on locations of impulse instants for the stabilization of impulse control
    system have not been touched upon. (At least, we do not find any such literature.) From this perspective,
     Theorem \ref{yu-theorem-5-13-1} is new.

    \item  Since controls affect \eqref{yu-10-24-1} and \eqref{421usualcontrolsystem} in different ways, we set up  \textbf{(I-I-LQ)} which differs from the usual LQ problem for the  control system \eqref{421usualcontrolsystem}. This modified LQ problem leads to a discrete dynamic programming principle,
       from which, we get the variant of Riccati's equation  \eqref{yu-11-26-1} and the  feedback law \eqref{yu-12-4-1-b}  
        differing from those for the usual control system \eqref{421usualcontrolsystem}.

\end{itemize}

\subsection{Related works}
\begin{itemize}

\item About the stabilization for impulse control systems,
we would like to mention \cite{b16, b18, b21, b17, b20, b10} and the references therein.

In \cite{b17}, the authors studied the stabilization for the system:
$$
x'(t)=Ax(t)+Bu(t),\;\;t\in\mathbb{R}^+\setminus\{t_j\}_{j\in\mathbb{N}^+},\;\;
    x(t_j^+)=A_Ix(t_j), j\in\mathbb{N}^+.
    $$
Under some assumption on time instants $\{t_j\}_{j\in\mathbb{N}^+}$, it was obtained that
if the above system has some reachable property, then it is  stabilizable.
This result  was generalized in \cite{b20} via another way.

In \cite{b21}, the author built up a Kalman-type controllability
    decomposition for the system:
$$
 x'(t)=Ax(t)+Bu(t), \;t\in\mathbb{R}^+\setminus\{t_j\}_{j\in\mathbb{N}^+},\;\;
    x(t_j^+)=A_Ix(t_j)+B_Iu_j, \;j\in\mathbb{N}^+.
$$
 Based on the decomposition, a necessary condition, as well as a sufficient condition, for the stabilization of the above system was given. Both results are related to some kind of reachability. The stabilization of the above system was also studied in \cite{b18}.

 In \cite{b10}, the authors studied the stabilization for impulse control heat equations.

 \item About the controllability for impulse control systems, we mention works: \cite{b15,b4,b5,b19,b2,b12,b14} and the references therein.

The work \cite{b19} concerns the  system:

$$
 x'(t)=Ax(t)+Bu(t),\; t\in\mathbb{R}^+\setminus\{t_j\}_{j\in\mathbb{N}^+},\;\;
    x(t_j^+)=A_Ix(t_j)+B_Iu_j,\;j\in\mathbb{N}^+.
$$
The authors gave links among reachable sets, some invariant subspaces and the feedback-reversibility.

In \cite{b2}, the authors studied the controllability for the system:
 $$
 x'(t)=Ax(t),\; t\in[0,T]\setminus\{t_j\}_{j=1}^q,\;\;
    \triangle x(t_j)=Bu_j,\;j\in\{1,2,\ldots,q\}.
    $$
(Here $T>0$, $q\in \mathbb{N}^+$ and $\{t_j\}_{j=1}^q\subset (0,T)$.)
They found  $q^{n,m}(A,B)\in \mathbb{N}^+$ (defined in (\ref{yu-3-17-4}) with $C=A$ and $D=B$) so that
for each $q\geq q^{n,m}(A,B)$ and each  $\{t_j\}_{j=1}^{q}\subset (0,T)$ with
$t_{q}-t_1<d_A$, the above system is controllable, provided that $(A,B)$ holds  Kalman controllability rank condition. This result is  used in the proofs of Theorem \ref{yu-theorem-3-14-1}, as well as Theorem \ref{yu-theorem-5-13-1}.

\item About optimal control for impulse control systems, we mention the works: \cite{b3,b8,b9,b11,b13}
and the references therein.

\item About general theory for impulse systems, we refer readers to \cite{b6,b7,b1} and the
references therein.
\end{itemize}

\subsection {Plan of this paper}
The rest of this paper is organized as follows: Section 2 proves Theorem \ref{yu-theorem-11-27-1}. Section 3 shows Theorem \ref{yu-theorem-3-14-1} and Theorem \ref{yu-theorem-5-13-1}. Section 4 gives conclusions and
perspectives.

\section{Proof of main results (Part I)}
The aim of this section is to prove Theorem \ref{yu-theorem-11-27-1}.
\subsection{On LQ problem}\label{section2}
               We arbitrarily fix a system $(A, \{B_k\}_{k=1}^{\hbar},\Lambda_{\hbar})$,
         $\mathcal{Q}=(Q_j)_{j\in\mathbb{N}^+}\in\mathfrak{M}^{n}_{\hbar,+}$ and
         $\mathcal{R}=(R_j)_{j\in\mathbb{N}^+}\in \mathfrak{M}^{m}_{\hbar,+}$.
          Recall   (\ref{yu-11-19-2}) and (\ref{yu-11-19-1}) for
         the definitions of  $\mathcal{U}_{ad}(x_0)$ and $J(u;x_0)$.
          Throughout this subsection, we assume
         \begin{eqnarray}\label{assum-admmi}
         \mathcal{U}_{ad}(x_0)\neq \emptyset\;\;\mbox{for all}\;\;x_0\in\mathbb{R}^n.
         \end{eqnarray}
         The value-function of the problem \textbf{(I-I-LQ)} is defined as:
      \begin{equation}\label{yu-3-8-1-b}
                       V(x_0):=\inf_{u\in\mathcal{U}_{ad}(x_0)}J(u;x_0),\; x_0\in \mathbb{R}^n.
         \end{equation}
        Because of \eqref{assum-admmi}, we have that $V(x_0)<+\infty$ for each $x_0\in \mathbb{R}^n$.
         From \eqref{yu-3-8-1-b}, \eqref{yu-11-19-2} and \eqref{yu-11-19-1}, one can directly check that
         $V(\cdot)$ is continuous and satisfies the parallelogram law:
         $$
         V(x_0+y_0)+V(x_0-y_0)=2(V(x_0)+V(y_0))\;\;\mbox{for all}\;\; x_0,y_0\in \mathbb{R}^n.
         $$
         (We omit the detailed proof here.) Then by \cite[Theorem 3]{b23}, we have

\begin{lemma}
There is a symmetric and positive definite matrix $P\in \mathbb{R}^{n\times n}$ so that
$V(x_0)=\langle Px_0, x_0\rangle$ for all $x_0\in \mathbb{R}^n$.
\end{lemma}

  Let  $x(\cdot;u,x_0,l)$, with $l\in\mathbb{N}$ and
 $x_0\in\mathbb{R}^n$,
be the solution to the equation:
\begin{equation*}
\begin{cases}
    x'(t)=Ax(t),&t\in(t_l,+\infty)\setminus\Lambda_{\hbar},\\
    \triangle x(t_j)=B_{\vartheta(j)}u_j,&j>l,\\
    x(t_l^+)=x_0.
\end{cases}
\end{equation*}
  We define, for each $x_0\in\mathbb{R}^n$ and each
   $l\in\mathbb{N}$,
\begin{equation}\label{definition of new admis}
    \mathcal{U}_{ad}(x_0;l):=\{u\in l^2(\mathbb{R}^m):(x(t_j;u,x_0,l))_{j>l}\in l^2(\mathbb{R}^n)\};
\end{equation}
    \begin{eqnarray}\label{yu-11-25-1}
    J(u;x_0,l):=\sum_{j=l+1}^{+\infty}\big[\langle Q_jx(t_j;u,x_0,l),x(t_j;u,x_0,l)\rangle+\langle
    R_ju_j,u_j\rangle\big],\; u\in \mathcal{U}_{ad}(x_0;l).
\end{eqnarray}
(They correspond to $ \mathcal{U}_{ad}(x_0)$ and $J(u;x_0)$ respectively.)
       One can easily check that
       \begin{eqnarray}\label{4.26,2.6}
       \mathcal{U}_{ad}(x_0;0)=\mathcal{U}_{ad}(x_0)\;\forall\; x_0\in \mathbb{R}^n
       \end{eqnarray}
       and that  for any $t>t_l$ (with $l\in \mathbb{N}$), $x_0\in\mathbb{R}^n$ and $u\in\mathcal{U}_{ad}(x_0;0)$,
\begin{equation}\label{yu-11-25-3}
    x(t;u,x_0)=x(t;u,x_0,0)=x(t;u,x(t_1^+;u,x_0),1)
    =\cdots=x(t;u,x(t_l^+;u,x_0),l).
\end{equation}

\par
   We now consider, for each $l\in\mathbb{N}$, the LQ problem \textbf{$(\mbox{I-I-LQ})_{l}$}: Given
   $x_0\in\mathbb{R}^n$,  find a control $u^*_l\in \mathcal{U}_{ad}(x_0;l)$ so that
\begin{equation}\label{4.24.2.6}
    V(x_0;l):=\inf_{u\in \mathcal{U}_{ad}(x_0;l)}J(u;x_0,l)=J(u^*_l;x_0,l).
\end{equation}
    It is clear that  \textbf{$(\mbox{I-I-LQ})_{0}$} coincides with  \textbf{$(\mbox{I-I-LQ})$}
    and $V(\cdot)=V(\cdot;0)$ (see (\ref{yu-3-8-1-b})).
    We call $V(\cdot;l)$ the value-function of \textbf{$(\mbox{I-I-LQ})_{l}$}.

\begin{lemma}\label{yu-lemma-3-12-1}
For any $l\in\mathbb{N}^+$ and $x_0\in\mathbb{R}^n$, it holds that
        $\mathcal{U}_{ad}(x_0;l)\neq \emptyset$ and $V(x_0;l)<+\infty$.
\end{lemma}
\begin{proof}
Arbitrarily fix $l\in\mathbb{N}^+$ and $x_0\in\mathbb{R}^n$. First of all,
we have  $\mathcal{U}_{ad}(x_0;0)\neq \emptyset\;\forall\; x_0\in\mathbb{R}^n$,
 because of   \eqref{assum-admmi} and
  \eqref{4.26,2.6}.
     We now claim
\begin{equation}\label{yu-3-12-1}
    \mathcal{U}_{ad}(x_0;N\hbar)\neq \emptyset \;\mbox{for any}\;x_0\in\mathbb{R}^n\;\mbox{and}\;N\in\mathbb{N}^+.
\end{equation}
  To this end, we arbitrarily fix  $N\in\mathbb{N}^+$.
  Define a map $\mathcal{H}: l^2(\mathbb{R}^m)\rightarrow l^2(\mathbb{R}^m)$ in the following manner: Given
  $u=(u_j)_{j\in \mathbb{N}^+}\in l^2(\mathbb{R}^m)$, let
  \begin{equation}\label{definitionofH,4.24}
     \mathcal{H}(u):=v=(v_j)_{j\in \mathbb{N}^+}\;\mbox{with}\; v_{j}=u_{j+N\hbar}\;\mbox{for all}\;j\in
     \mathbb{N}^+.
\end{equation}
   By \eqref{definitionofH,4.24}, \eqref{timeinstants} and \eqref{index1}, we can directly check that
  for each $u\in l^2(\mathbb{R}^m)$,
  \begin{eqnarray}\label{4.23-2.8}
  x(t_{j+N\hbar};{u},x_0,N\hbar)=x(t_j; \mathcal{H}(u),x_0,0)\;\;\mbox{for all}\;\;j\in \mathbb{N}^+.
  \end{eqnarray}
 By  \eqref{4.23-2.8} and \eqref{definition of new admis}, we can easily find
  \begin{eqnarray}\label{4.26,equiv.}
  u\in \mathcal{U}_{ad}(x_0;N\hbar)\Leftrightarrow \mathcal{H}(u)\in \mathcal{U}_{ad}(x_0;0).
  \end{eqnarray}
  Since $\mathcal{H}$ is surjective and $\mathcal{U}_{ad}(x_0;0)\neq\emptyset$, we get
  \eqref{yu-3-12-1} from \eqref{4.26,equiv.}.

\par

Next, we let
    $N=\left[{l}/{\hbar}\right]$ (which implies
    $N\hbar< l\leq(N+1)\hbar$).
    By (\ref{yu-3-12-1}), we can take
    \begin{eqnarray}\label{4.23-2.9}
    \hat{v}=\{\hat{v}_j\}_{j\in \mathbb{N}^+}\in\mathcal{U}_{ad}(e^{A(t_{(N+1)\hbar}-t_l)}x_0;(N+1)\hbar).
    \end{eqnarray}
    Define
   $\hat{u}=(\hat{u}_j)_{j\in\mathbb{N}^+}$ in the manner:
   $\hat{u}_j:=0$, when $1\leq j\leq (N+1)\hbar$; $\hat{u}_j:=\hat{v}_j$, when
   $j>(N+1)\hbar$.
   Then by  (\ref{yu-11-25-3}), we see
$$
    x(t_j;\hat{u},x_0,l)=
\begin{cases}
    e^{A(t_j-t_l)}x_0,&\mbox{if}\;l\leq j\leq (N+1)\hbar,\\
    x(t_j;\hat{v},e^{A(t_{(N+1)\hbar}-t_j)}x_0,(N+1)\hbar),&\mbox{if}\;j>(N+1)\hbar.
\end{cases}
$$
    This, along with \eqref{4.23-2.9}, yields $(x(t_j;\hat{u},x_0,l))_{j>l}\in l^2(\mathbb{R}^n)$ which implies
    $\hat{u}\in\mathcal{U}_{ad}(x_0;l)$.
    So $\mathcal{U}_{ad}(x_0;l)\neq \emptyset$, which, along with \eqref{4.24.2.6},
    shows that $V(x_0;l)<+\infty$.
    This ends the proof. \end{proof}

    By Lemma \ref{yu-lemma-3-12-1}, we see that  $V(x_0;l)<+\infty$ for all $l\in \mathbb{N}$
    and $x_0\in \mathbb{R}^n$.

\begin{lemma}\label{yu-lemma-3-13-1}
\begin{enumerate}
  \item[(i)] For each $l\in\mathbb{N}$, there is a symmetric and positive definite matrix
    $P_l\in \mathbb{R}^{n\times n}$ such that $V(x_0;l)=\langle P_lx_0,x_0\rangle$ for any
    $x_0\in\mathbb{R}^n$;
  \item[(ii)]  It holds that  $P_{l+\hbar}=P_l$ for all $l\in \mathbb{N}$.
\end{enumerate}
\end{lemma}
\begin{proof}
   \emph{The proof of the claim $(i)$.} One can use Lemma
   \ref{yu-lemma-3-12-1} to see that  for each $l\in \mathbb{N}$, $V(\cdot;l)$ is continuous and satisfies the
parallelogram law.
          Then the desired result follows from  \cite[Theorem 3]{b23}.

\vskip 5pt
    \emph{The proof of the claim $(ii)$.} By the claim $(i)$ of this lemma, we see that it suffices to  show that, for each $l\in\mathbb{N}$,
\begin{equation}\label{yu-3-13-7}
    V(x_0;l)=V(x_0;l+\hbar)\;\mbox{for any}\;x_0\in\mathbb{R}^n.
\end{equation}
We only show \eqref{yu-3-13-7} for the case that $l=0$, i.e.,
\begin{equation}\label{yu-2-26-2}
    V(x_0;0)=V(x_0;\hbar)\;\mbox{for any}\;x_0\in\mathbb{R}^n.
\end{equation}
   The general cases can be proved by the same way.
   To prove \eqref{yu-2-26-2}, we arbitrarily fix $x_0\in\mathbb{R}^n$.
   Define a map $\mathcal{H}_\hbar: l^2(\mathbb{R}^m)\rightarrow l^2(\mathbb{R}^m)$ by
   \begin{equation}\label{yu-2-27-1}
    \mathcal{H}_\hbar(u):=v\;\;\mbox{with}\;v_{j}=u_{j+\hbar}\;\mbox{for each}\;j\in \mathbb{N}^+.
\end{equation}
   Then $\mathcal{H}_\hbar$ is surjective.
      By \eqref{yu-2-27-1}, \eqref{timeinstants} and \eqref{index1}, we can directly check that for each $u\in
      l^2(\mathbb{R}^m)$,
      \begin{eqnarray}\label{4.24,2.15}
   x(t_{j+\hbar};u,x_0,\hbar)=x(t_j; \mathcal{H}_\hbar(u),x_0,0)\;\;\mbox{for all}\;\;j\in \mathbb{N}^+.
   \end{eqnarray}
  From \eqref{yu-2-27-1}, \eqref{4.24,2.15} and \eqref{definition of new admis}, we find
  \begin{eqnarray}\label{4.24,2.17}
  \mathcal{H}_\hbar(u)\in\mathcal{U}_{ad}(x_0;0) \Leftrightarrow u\in\mathcal{U}_{ad}(x_0;\hbar).
  \end{eqnarray}
  Since $Q_{j+\hbar}=Q_j$ and
    $R_{j+\hbar}=R_j$ for each $j\in\mathbb{N}^+$, we see from
   \eqref{4.24,2.17} and \eqref{4.24,2.15} that
  \begin{eqnarray}\label{4.24,2.18}
  J(u;x_0,\hbar)=J(\mathcal{H}_\hbar(u);x_0,0), \;\mbox{when}\;\; u\in\mathcal{U}_{ad}(x_0;\hbar).
  \end{eqnarray}
  By  \eqref{4.24,2.17} and  \eqref{4.24,2.18}, we find
  \begin{eqnarray*}
  V(x_0;0)\leq J(\mathcal{H}_\hbar(u);x_0,0)=J(u;x_0,\hbar) \;\mbox{for each}\;\;
  u\in\mathcal{U}_{ad}(x_0;\hbar),
  \end{eqnarray*}
  which, together with
  \eqref{4.24.2.6}, leads to
     \begin{equation}\label{yu-2-27-2}
    V(x_0;0)\leq V(x_0;\hbar).
\end{equation}

    We next show the reverse of \eqref{yu-2-27-2}.
    By \eqref{4.24.2.6}, we can find, for
    each $\varepsilon>0$, a control $v_\varepsilon\in \mathcal{U}(x_0;0)$
        so that
     \begin{eqnarray}\label{4.25,2.22}
      V(x_0;0)+\varepsilon\geq J(v_\varepsilon;x_0,0).
    \end{eqnarray}
        Since $\mathcal{H}_\hbar$ is surjective, there is $u_\varepsilon\in l^2(\mathbb{R}^m)$ so that
      $\mathcal{H}_{\hbar}(u_\varepsilon)=v_\varepsilon$.
      This, along with \eqref{4.24,2.17}, leads to
       \begin{eqnarray}\label{4.25,2.21}
     u_\varepsilon\in \mathcal{U}_{ad}(x_0;\hbar).
    \end{eqnarray}
      From \eqref{4.25,2.21} and \eqref{4.24,2.18}, we find that
        $J(u_\varepsilon;x_0,\hbar)=J(v_\varepsilon;x_0, 0)$.
               This, together with \eqref{4.25,2.22} and \eqref{4.24.2.6}, yields
        $$
    V(x_0;0)+\varepsilon\geq J(v_\varepsilon;x_0,0)
    =J(u_\varepsilon;x_0,\hbar)\geq V(x_0;\hbar).
$$
   Sending  $\varepsilon\rightarrow 0$ in the above gives $V(x;0)\geq V(x;\hbar)$, which, together with
   (\ref{yu-2-27-2}), leads to (\ref{yu-2-26-2}).
   This ends the proof.
\end{proof}

The next Lemma \ref{yu-lemma-11-25-3} gives a discrete dynamic programming principle
associated to \textbf{$(\mbox{I-I-LQ})_{l}$}.
Throughout this lemma and its proof, we will use the notations:
 For each $v\in l^2(\mathbb{R}^m)$, $x_0\in \mathbb{R}^n$, $l\in \mathbb{N}$, we let
\begin{eqnarray}\label{notation4.25,2.25}
 q(j;v,x_0,l):=\langle Q_jx(t_j;v,x_0,l),x(t_j;v,x_0,l)\rangle+\langle R_jv_j,v_j\rangle,\; j=l+1, l+2,\dots.
\end{eqnarray}
For each  $w=(w_1,\cdots,w_k)$
(with $w_j\in \mathbb{R}^{m}$, $k\in \mathbb{N}^+$) and each $(v_j)_{j\in \mathbb{N}^+}\in
l^2(\mathbb{R}^m)$, we write
\begin{eqnarray}\label{426,2.27notation1}
w \odot v:=(w_1,\cdots,w_k, v_1,v_2,\cdots).
\end{eqnarray}
For each $v=(v_j)_{j\in \mathbb{N}^+}\in l^2(\mathbb{R}^m)$ and each $k\in \mathbb{N}^+$, we let
\begin{eqnarray}\label{426,2.27notation2}
E_k(v):=(v_j)_{j=1}^k\in \mathbb{R}^{m\times k}\;\;\mbox{and}\;\; G_k(v):=(v_{k+j})_{j\in \mathbb{N}^+}\in
l^2(\mathbb{R}^m).
\end{eqnarray}

\begin{lemma}\label{yu-lemma-11-25-3}
With notations in \eqref{notation4.25,2.25} and \eqref{426,2.27notation1}, it holds that for each  $l\in\mathbb{N}$ and each $k>l$,
\begin{equation*}
    V(x_0;l)=\inf_{w\in \mathbb{R}^{m\times k}}\Big\{\sum_{j=l+1}^kq(j;  w\odot 0, x_0,l)+V(x(t_k^+; w\odot
    0,x_0,l);k)\Big\}\;\;\mbox{for each}\;\;x_0\in \mathbb{R}^n.
\end{equation*}
Here, $0$ is the origin of $l^2(\mathbb{R}^m)$.
    \end{lemma}
    \begin{proof}
    Arbitrarily fix  $x_0\in \mathbb{R}^n$, $l$ and $k$ with $k>l$.
    By Lemma \ref{yu-lemma-3-12-1}, we have $\mathcal{U}_{ad}(x_0;l)\neq \emptyset$.
        We organize the rest of the proof by two steps.

    \noindent {\emph{Step 1.} We prove
    \begin{equation}\label{yu-11-25-14}
    V(x_0;l)\leq \inf_{w\in \mathbb{R}^{m\times k}}\Big\{\sum_{j=l+1}^kq(j;w\odot 0,x_0,l)+V(x(t_k^+;w\odot
    0,x_0,l);k)\Big\}.
\end{equation}     }

To show \eqref{yu-11-25-14}, it suffices to prove
\begin{equation}\label{yu-11-25-13}
    V(x_0;l)\leq \sum_{j=l+1}^kq(j;E_k(v)\odot 0,x_0,l)+V(x(t_k^+;E_k(v)\odot 0,x_0,l);k)
    \;\forall\; v\in E_k(\mathcal{U}_{ad}(x_0;l))
\end{equation}
and
\begin{equation}\label{yu-3-14-5}
    E_k(\mathcal{U}_{ad}(x_0;l))=\mathbb{R}^{m\times k}.
\end{equation}

We first show \eqref{yu-11-25-13}. Arbitrarily fix
    $v\in \mathcal{U}_{ad}(x;l)$.
    We can directly check the following facts:
\begin{equation}\label{yu-11-25-8}
    V(x_0;l)\leq \sum_{j=l+1}^kq(j;v,x_0,l)+\sum_{j=k+1}^{+\infty}q(j;v,x_0,l);
\end{equation}
\begin{equation}\label{yu-11-25-10}
    \sum_{j=k+1}^{+\infty}q(j;v,x,l)=J(v;x(t^+_k;v,x,l),k)\;\;\mbox{for each}\;\; j\geq k+1;
\end{equation}
   \begin{equation}\label{yu-11-25-11}
    J(E_k(v)\odot G_k(v);x,k)=J(\hat{v}^k\odot G_k(v);x,k)\;\;\mbox{for any}\;\; \hat{v}^k\in
    \mathbb{R}^{m\times k}.
\end{equation}
We now claim
   \begin{eqnarray}\label{4.28,2.35}
   E_k(v)\odot G_k(u)\in \mathcal{U}_{ad}(x;l),\;\;\mbox{when }\;\;u\in\mathcal{U}_{ad}(x(t_k^+;v,x,l);k).
   \end{eqnarray}
  To this end, we arbitrarily fix $u\in\mathcal{U}_{ad}(x(t_k^+;v,x,l);k)$. Then from  \eqref{426,2.27notation2}, we have
      \begin{eqnarray*}
    x(t_j;E_k(v)\odot G_k(u),x_0,l)=
    \begin{cases}
    x(t_j;E_k(v)\odot 0,x_0,l), &\mbox{if}\; l+1\leq j\leq k,\\
    e^{A(t_{j}-t_{j-1})}x(t_{j-1}^+;u,x(t_k^+;v,x_0,l),t_k), &\mbox{if}\; j> k.
    \end{cases}
   \end{eqnarray*}
   Meanwhile, we  can directly check
      \begin{eqnarray*}
   (x(t_{j-1}^+;u,x(t_k^+;v,x_0,l),t_k))_{j> k}
   =(e^{-A(t_{j}-t_{j-1})}x(t_{j}; u,x(t_k^+;v,x_0,l),t_k))_{j> k}.
   \end{eqnarray*}
  From these, \eqref{4.28,2.35} follows.

   By \eqref{4.28,2.35},
    (\ref{yu-11-25-10}), (\ref{yu-11-25-11}) and by (\ref{yu-11-25-8})
    (where  $v$ is replaced by $E_k(v)\odot G_k(u)$),
   we find
$$
    V(x_0;l)\leq \sum_{j=l+1}^kq(j;E_k(v)\odot 0,x_0,l)+J(u;x(t_k^+;E_k(v)\odot 0,x_0,l),k)
    \;\forall\; u\in\mathcal{U}_{ad}(x(t_k^+;v,x_0,l);k),
$$
    which leads to \eqref{yu-11-25-13}.  Here, we have used the facts:
$$
    x(t_k^+;E_k(v)\odot G_k(u),x_0,l)=x(t_k^+;E_k(v)\odot 0,x_0,l)
$$
    and
$$
    \sum_{j=l+1}^kq(j;E_k(v)\odot G_k(u),x_0,l)=\sum_{j=l+1}^kq(j;E_k(v)\odot 0,x_0,l).
$$

    We next show \eqref{yu-3-14-5}. In fact, it follows  by \eqref{426,2.27notation2} that
    $E_k(\mathcal{U}_{ad}(x_0;l))\subset\mathbb{R}^{m\times k}$.
    Conversely, for each $w=(w_1,\cdots,w_{k})\in\mathbb{R}^{m\times k}$,
    we take  $\hat u\in\mathcal{U}_{ad}(x(t_{k}^+;w\odot 0,x_0,l);k)$.
     Let  $\hat w:=w\odot \hat u$. Then by \eqref{4.28,2.35} and \eqref{426,2.27notation2}, we find
     $\hat w\in \mathcal{U}_{ad}(x_0;l)$ and  $E_k(\hat w)=w$. Hence, $\mathbb{R}^{m\times k}\subset
     E_k(\mathcal{U}_{ad}(x_0;l))$. So \eqref{yu-3-14-5} is true.

\par
\noindent {\emph{Step 2.} We prove
\begin{equation}\label{4.29.2.36}
    V(x_0;l)\geq \inf_{w\in \mathbb{R}^{m\times k}}\Big\{\sum_{j=l+1}^kq(j;w\odot 0,x_0,l)+V(x(t_k^+;w\odot
    0,x_0,l);k)\Big\}.
\end{equation} }

It follows by \eqref{4.24.2.6}, \eqref{yu-11-25-1} and \eqref{notation4.25,2.25}   that  for each
$\varepsilon>0$, there is  $v^\varepsilon\in \mathcal{U}_{ad}(x_0;l)$ such that
\begin{eqnarray}\label{yu-11-25-15}
    V(x_0;l)+\varepsilon&\geq& \sum_{j=l+1}^kq(j;v^\varepsilon,x_0,l)+\sum_{j=k+1}^{+\infty}
    q(j;v^\varepsilon,x_0,l)\nonumber\\
        &\geq&\sum_{j=l+1}^kq(j;E_k(v^\varepsilon)\odot 0,x_0,l)+V(x(t_k^+;E_k(v^\varepsilon)\odot
        0,x_0,l);k).
\end{eqnarray}
    Here, we have used facts:
    \begin{equation*}
    x(t_k^+;v^\varepsilon,x_0,l)=x(t_k^+;E_k(v^\varepsilon)\odot 0,x_0,l);\;
    \;\sum_{j=l+1}^kq(j;v^\varepsilon,x_0,l)=\sum_{j=l+1}^kq(j;E_k(v^\varepsilon)\odot 0,x_0,l);
\end{equation*}
and
\begin{equation*}
v^\varepsilon\in \mathcal{U}_{ad}(x(t_k^+, E_k(v^\varepsilon)\odot 0, x_0,l); k).
\end{equation*}
(The last fact above holds, since  $v^\varepsilon\in \mathcal{U}_{ad}(x_0;l)$.)
    From   \eqref{yu-11-25-15} and  (\ref{yu-3-14-5}), we see
\begin{equation*}\label{yu-11-25-16}
    V(x_0;l)+\varepsilon\geq \inf_{w\in \mathbb{R}^{m\times k}}\Big\{\sum_{j=l+1}^kq(j;w\odot
    0,x_0,l)+V(x(t_k^+;w\odot 0,x_0,l);k)\Big\}.
\end{equation*}
    Letting $\varepsilon\to 0^+$ in the above leads to \eqref{4.29.2.36}.

    Thus, by (\ref{yu-11-25-14}) and (\ref{4.29.2.36}), we end the proof.
    \end{proof}

\subsection{Proof of Theorem \ref{yu-theorem-11-27-1}}

Arbitrarily fix $(A,\{B_k\}_{k=1}^\hbar,\Lambda_{\hbar})$. First of all, we give the following two statements:
\begin{enumerate}
\item[$(iii)'$] For any $\mathcal{Q}\in \mathfrak{M}^{n}_{\hbar,+}$
and $\mathcal{R}\in \mathfrak{M}^{m}_{\hbar,+}$, the equation (\ref{yu-11-26-1}) has a solution
$\{P_{k}\}_{k=0}^{\hbar}$.
  \item[$(iv)'$] There is $\mathcal{Q}\in \mathfrak{M}^{n}_{\hbar,+}$
and $\mathcal{R}\in \mathfrak{M}^{m}_{\hbar,+}$ so that the equation (\ref{yu-11-26-1}) has a solution
$\{P_{k}\}_{k=0}^{\hbar}$.
\end{enumerate}
We will prove $(i)\Rightarrow (ii)\Rightarrow (iii)'\Rightarrow (iv)'\Rightarrow (i)$
and $(iii)'\Rightarrow (iii)$. When these are done, we finish the proof of Theorem \ref{yu-theorem-11-27-1}, since it is clear that $(iii)\Rightarrow (iv)$ and $(iv)\Rightarrow (iv)'$.

We organize the proof by several steps.

    \vskip 5pt

    \noindent {\it Step 1. We prove $(i)\Rightarrow(ii)$.}

     Suppose that $(i)$ is true. Then there is   $\mathcal{F}:=\{F_k\}_{k=1}^{\hbar}\subset\mathbb{R}^{m\times n}$
     so that \eqref{stbleexpression} is true. Arbitrarily fix $x_0\in \mathbb{R}^n$ and $\mathcal{F}$ so that \eqref{stbleexpression} holds.
     Write $x_{\mathcal{F}}(\cdot; x_0)$ for the solution to the equation \eqref{yu-10-24-4} with the initial
     condition:
     $x(0)=x_0$. Take control  $u:=(u_{j})_{j\in\mathbb{N}^+}$ with $u_j=F_{\vartheta(j)}x_{\mathcal{F}}(t_j; x_0)$, $j\in
     \mathbb{N}^+$. Then we have
     $x(t; u,x_0)=x_{\mathcal{F}}(t; x_0)$ for $t\geq 0$. This, along with \eqref{stbleexpression}, indicates
     that $(x(t_j;u,x_0))_{j\in\mathbb{N}^+}\in l^2(\mathbb{R}^n)$
     and $u\in l^2(\mathbb{R}^m)$. Thus, $\mathcal{U}_{ad}(x_0)\neq \emptyset$. So $(ii)$ holds.

\vskip 5pt

\noindent {\emph{Step 2. We prove $(ii)\Rightarrow(iii)'$.}}

Suppose that $(ii)$ is true. Arbitrarily fix $\mathcal{Q}\in\mathfrak{M}_{\hbar,+}^{n}$ and
$\mathcal{R}\in\mathfrak{M}_{\hbar,+}^{m}$. Let $\mathcal{P}=(P_{l})_{l\in\mathbb{N}}$
be given by  Lemma \ref{yu-lemma-3-13-1}. (Notice that Lemma \ref{yu-lemma-3-13-1} needs the assumption
\eqref{assum-admmi} which is exactly $(ii)$.)
We will show that  $\{P_k\}_{k=0}^\hbar$ is a solution to  the  equation  (\ref{yu-11-26-1}).

 First, we  show that
   $\{P_k\}_{k=0}^\hbar$ satisfies the first  equation in  (\ref{yu-11-26-1}), i.e.,  for each $0\leq l<\hbar$,
\begin{equation}\label{yu-3-14-8}
    e^{-A^\top(t_{l+1}-t_l)}P_le^{-A(t_{l+1}-t_l)}-P_{l+1}
    =Q_{l+1}-P_{l+1}B_{l+1}(R_{l+1}+B_{l+1}^\top P_{l+1}B_{l+1})^{-1}B_{l+1}^\top P_{l+1}.
\end{equation}
To this end, we arbitrarily fix $x_0\in\mathbb{R}^n$ and $0\leq l<\hbar$.
It follows by  Lemma \ref{yu-lemma-3-13-1} and Lemma \ref{yu-lemma-11-25-3} that for any $v=(v_1,\cdots,v_{l+1})$ (with $v_j\in\mathbb{R}^{m}$ for all $j$),
\begin{eqnarray}\label{yu-11-29-5-b}
    &\;&\langle P_lx_0,x_0\rangle-\langle P_{l+1} x(t_{l+1}^+;v\odot 0,x_0,l),x(t_{l+1}^+;v\odot
    0,x_0,l)\rangle\nonumber\\
    &\leq&\langle Q_{l+1}x(t_{l+1};v\odot 0,x_0,l),x(t_{l+1};v\odot 0,x_0,l)\rangle
+\langle R_{l+1}v_{l+1},v_{l+1}\rangle.
\end{eqnarray}
	(Here $0$ is the origin of $l^2(\mathbb{R}^m)$.)
 Meanwhile, one can directly check that for any $v=(v_1,\cdots,v_{l+1})$ (with $v_j\in\mathbb{R}^{m}$ for all $j$),
\begin{eqnarray}\label{3.392019616}
	&\;&\langle P_{l+1} x(t_{l+1}^+;v\odot 0,x_0,l),x(t_{l+1}^+;v\odot 0,x_0,l)\rangle\nonumber\\
	&=&\langle e^{A^\top(t_{l+1}-t_l)}P_{l+1} e^{A(t_{l+1}-t_l)}x_0,x_0\rangle
	+2\langle B_{l+1}^\top P_{l+1} e^{A(t_{l+1}-t_l)}x_0,v_{l+1}\rangle\nonumber\\
	&\;&+\langle B_{l+1}^\top P_{l+1} B_{l+1}v_{l+1},v_{l+1}\rangle
\end{eqnarray}
and
   \begin{equation}\label{3.402019616}
	\langle Q_{l+1}x(t_{l+1};v\odot 0,x_0,l),x(t_{l+1};v\odot 0,x_0,l)\rangle=\langle
e^{A^\top(t_{l+1}-t_l)}Q_{l+1}e^{A(t_{l+1}-t_l)}x_0,x_0\rangle.
\end{equation}
	These, together with (\ref{yu-11-29-5-b}), imply that for any $v=(v_1,\cdots,v_{l+1})$
(with $v_j\in\mathbb{R}^{m}$ for all $j$),
\begin{eqnarray*}
	&\;&\left\langle\left[ P_l-e^{A^\top(t_{l+1}-t_l)}\left(P_{l+1}+Q_{l+1}\right)
    e^{A(t_{l+1}-t_l)}\right]x_0,x_0\right\rangle\nonumber\\
	&\leq&\left\langle\left(R_{l+1}+B_{l+1}^\top P_{l+1}B_{l+1}\right)v_{l+1},v_{l+1}\right\rangle
	+2\langle B_{l+1}^\top P_{l+1} e^{A(t_{l+1}-t_l)}x_0,v_{l+1}
    \rangle\nonumber\\
    &=&\left\|\left(R_{l+1}+B_{l+1}^\top P_{l+1}B_{l+1}\right)^{\frac{1}{2}}\left[v_{l+1}
    +\left(R_{l+1}+B_{l+1}^\top P_{l+1}B_{l+1}\right)^{-1}
    B_{l+1}^\top P_{l+1}e^{A(t_{l+1}-t_l)}x_0\right]\right\|^2_{\mathbb{R}^n}\nonumber\\
	&\;&-\left\langle e^{A^\top(t_{l+1}-t_l)}P_{l+1} B_{l+1}\left(R_{l+1}+B_{l+1}^\top P_{l+1}
B_{l+1}\right)^{-1}B_{l+1}^\top P_{l+1} e^{A(t_{l+1}-t_l)}x_0,x_0\right\rangle.
\end{eqnarray*}
	Letting $v_{l+1}:=-\left(R_{l+1}+B_{l+1}^\top P_{l+1}B_{l+1}\right)^{-1}B_{l+1}^\top
P_{l+1}e^{A(t_{l+1}-t_l)}x_0$ in the above leads to
\begin{eqnarray}\label{yu-11-30-4-b}
	&\;&\left\langle\left[ P_l-e^{A^\top(t_{l+1}-t_l)}\left(P_{l+1}
    +Q_{l+1}\right)e^{A(t_{l+1}-t_l)}\right]x_0,x_0\right\rangle\nonumber\\
	&\leq&-\left\langle e^{A^\top(t_{l+1}-t_l)}P_{l+1} B_{l+1}\left(
     R_{l+1}+B_{l+1}^\top P_{l+1} B_{l+1}\right)^{-1}B_{l+1}^\top P_{l+1}
     e^{A(t_{l+1}-t_l)}x_0,x_0\right\rangle.
\end{eqnarray}
	
On the other hand, by Lemmas \ref{yu-lemma-3-13-1}, \ref{yu-lemma-11-25-3},
  for each $\varepsilon>0$, there is   $v^\varepsilon\in \mathbb{R}^{m\times(l+1)}$ so that
\begin{eqnarray*}\label{yu-h-11-30-5}
	&\;&\langle P_lx_0,x_0\rangle-\langle P_{l+1} x(t_{l+1}^+;v^\varepsilon\odot
0,x,l),x(t_{l+1}^+;v^\varepsilon\odot 0,x_0,l)\rangle
    +\varepsilon\nonumber\\
        &\geq& \langle Q_{l+1} x(t_{l+1};v^\varepsilon\odot 0,x_0,l),x(t_{l+1};v^\varepsilon\odot
        0,x_0,l)\rangle+\langle R_{l+1}v^\varepsilon_{l+1},v^\varepsilon_{l+1}\rangle.
\end{eqnarray*}
	This, along with \eqref{3.392019616} and \eqref{3.402019616} (where $v=v^\varepsilon$), yields
  \begin{eqnarray*}\label{yu-h-11-30-7}	
&\;&\left\langle\left[P_l-e^{A^\top(t_{l+1}-t_l)}\left(P_{l+1}+Q_{l+1}\right)
    e^{A(t_{l+1}-t_l)}\right]x_0,x_0\right\rangle+\varepsilon\nonumber\\
	&\geq&-\left\langle e^{A^\top(t_{l+1}-t_l)}P_{l+1} B_{l+1}\left(R_{l+1}+B_{l+1}^\top P_{l+1}
B_{l+1}\right)^{-1}B_{l+1}^\top P_{l+1} e^{A(t_{l+1}-t_l)}x_0,x_0\right\rangle.
\end{eqnarray*}
Sending 	$\varepsilon\rightarrow 0$ in the above, then combining (\ref{yu-11-30-4-b}), we obtain
(\ref{yu-3-14-8}) by the arbitrariness of $x_0$.

Besides, by the conclusion $(ii)$ in Lemma \ref{yu-lemma-3-13-1}, we see that $P_0=P_\hbar$, i.e.,
$\{P_k\}_{k=0}^\hbar$ satisfies the second  equation in  (\ref{yu-11-26-1}). So
$\{P_k\}_{k=0}^\hbar$ is a solution to  the  equation  (\ref{yu-11-26-1}).

\vskip 5pt

\noindent {\it Step 3. It is trivial that  $(iii)'\Rightarrow (iv)'$.}

\vskip 5pt

\noindent {\it Step 4. We prove that $(iv)'\Rightarrow (i)$.}

Suppose that  $\mathcal{Q}$, $\mathcal{R}$ and
$\{P_{k}\}_{k=0}^{\hbar}$ are given by $(iv)'$.
 Then we can find positive constants $C_{\min}$, $C_{\max}$ and $\widehat{C}$ so that
\begin{equation}\label{yu-5-11-1}
    C_{\min}\mathbb{I}_n\leq P_k\leq C_{\max}\mathbb{I}_n\;\forall\; k\in\{0,1,\ldots,\hbar\}
    \;\;\mbox{and}\;\;Q_j\geq \widehat{C}\mathbb{I}_n\;\forall\; j\in\mathbb{N}^+.
\end{equation}
(\eqref{yu-5-11-1} will be used later.)
      Let $\mathcal{F}:=\{F_k\}_{k=1}^\hbar$ be the corresponding feedback law given
    by (\ref{yu-12-4-1-b}).

    We claim that the corresponding closed-loop system (\ref{yu-10-24-4}) is stable.
    For this purpose, we arbitrarily fix a solution $x_\mathcal{F}(\cdot)$   to (\ref{yu-10-24-4}). Our aim is to show that it satisfies \eqref{stbleexpression}. The proof is divided by two parts.
    \vskip 5pt

   \noindent  {\it Part 4.1. We prove that for some $\mu>0$,
    \begin{equation}\label{yu-5-9-10}
    \langle P_{\hbar}x_{\mathcal{F}}(t^+_{\hbar}),x_{\mathcal{F}}(t^+_{\hbar})\rangle \leq e^{-\mu
    t_{\hbar}}\langle P_{\hbar}x_{\mathcal{F}}(0),x_{\mathcal{F}}(0)\rangle.
\end{equation}}

    Since
    $$
    x_{\mathcal{F}}(t_k^+)=e^{A(t_k-t_{k-1})}x_{\mathcal{F}}(t^+_{k-1})+B_kF_k
    e^{A(t_k-t_{k-1})}x_{\mathcal{F}}(t_{k-1}^+)\;\forall\; k\in \{1,2,\ldots,\hbar\},
    $$
    one can directly see from (\ref{yu-11-26-1}) and (\ref{yu-12-4-1-b}) that
    for each $k\in \{1,2,\ldots,\hbar\}$,
    \begin{eqnarray}\label{2.45201967}
    &\;&\langle P_kx_{\mathcal{F}}(t_k^+),x_{\mathcal{F}}(t_k^+)\rangle\nonumber\\
        &=&\langle e^{A^\top (t_k-t_{k-1})}P_ke^{A(t_k-t_{k-1})}x_{\mathcal{F}}(t_{k-1}^+),
    x_{\mathcal{F}}(t_{k-1}^+)\rangle\nonumber\\
    &\;&-\langle e^{A^\top (t_k-t_{k-1})}P_kB_k(R_k+B_k^\top P_kB_k)^{-1}B_k^\top P_ke^{A(t_k-t_{k-1})}
    x_{\mathcal{F}}(t_{k-1}^+), x_{\mathcal{F}}(t_{k-1}^+)\rangle\nonumber\\
    &\;&-\langle R_kF_ke^{A(t_k-t_{k-1})}x_{\mathcal{F}}(t_{k-1}^+),
    F_ke^{A(t_k-t_{k-1})}x_{\mathcal{F}}(t_{k-1}^+)\rangle\nonumber\\
    &\leq&\langle P_{k-1}x_{\mathcal{F}}(t_{k-1}^+),x_{\mathcal{F}}(t_{k-1}^+)\rangle-\langle
    Q_ke^{A(t_k-t_{k-1})}x_{\mathcal{F}}(t^+_{k-1}),
    e^{A(t_k-t_{k-1})}x_{\mathcal{F}}(t^+_{k-1})\rangle.
\end{eqnarray}
Meanwhile, by  (\ref{yu-5-11-1}), we find  that
    for each $k\in \{1,2,\ldots,\hbar\}$,
\begin{eqnarray}\label{2.46201967}
    &\;&\langle Q_ke^{A(t_k-t_{k-1})}x_{\mathcal{F}}(t_{k-1}^+),e^{A(t_k-t_{k-1})}x_{\mathcal{F}}(t_{k-1}^+)\rangle\nonumber\\
    &\geq& \widehat{C}C_{\max}^{-1}\Big[\sup_{s\in[0,t_{\hbar}]}\|e^{-As}\|_{\mathcal{L}(\mathbb{R}^n)}\Big]^{-2}
    \langle P_{k-1}x_{\mathcal{F}}(t_{k-1}^+),x_{\mathcal{F}}(t_{k-1}^+)\rangle.
\end{eqnarray}
Now by letting $\rho:=1-\widehat{C}C_{\max}^{-1}\Big[\sup_{s\in[0,t_{\hbar}]}\|e^{-As}\|_{\mathcal{L}(\mathbb{R}^n)}\Big]^{-2}$, we obtain from  \eqref{2.45201967} and \eqref{2.46201967} that
    \begin{eqnarray}\label{yu-5-10-1}
    \langle P_kx_{\mathcal{F}}(t_k^+),x_{\mathcal{F}}(t_k^+)\rangle\leq \rho
    \langle P_{k-1}x_{\mathcal{F}}(t_{k-1}^+),x_{\mathcal{F}}(t_{k-1}^+)\rangle \;\forall\; k\in \{1,2,\ldots,\hbar\},
\end{eqnarray}
    which leads to
\begin{eqnarray}\label{2.48201967}
    \langle P_{\hbar}x_{\mathcal{F}}(t_{\hbar}^+),x_{\mathcal{F}}(t_{\hbar}^+)\rangle
    \leq \rho^{\hbar}\langle P_0x_{\mathcal{F}}(0),x_{\mathcal{F}}(0)\rangle=\rho^{\hbar}\langle P_\hbar
    x_{\mathcal{F}}(0),x_{\mathcal{F}}(0)\rangle.
\end{eqnarray}

  Notice that $0\leq\rho<1$. (This follows from  \eqref{yu-5-10-1}.)
  In the case that  $\rho=0$, we see from \eqref{2.48201967} that
   (\ref{yu-5-9-10}) holds for any $\mu>0$. In the case when $\rho\in(0,1)$,
   we see from \eqref{2.48201967} that
   (\ref{yu-5-9-10}) holds for $\mu=-\frac{\hbar}{t_{\hbar}}\ln\rho$.
   Hence, \eqref{yu-5-9-10} has been proved.

   \vskip 5pt
   \noindent {\it Part 4.2. We prove that $x_\mathcal{F}(\cdot)$  satisfies  \eqref{stbleexpression}.    }

Since $jt_{\hbar}=t_{j\hbar}$ for all $j\in \mathbb{N}^+$ (see \eqref{1.5201967}), it follows from
 (\ref{yu-5-9-10}) that
$$
    \langle P_{\hbar}x_{\mathcal{F}}(t^+_{j\hbar}),x_{\mathcal{F}}(t^+_{j\hbar})\rangle \leq e^{-\mu
    t_{j\hbar}}\langle P_{\hbar}x_{\mathcal{F}}(0),x_{\mathcal{F}}(0)\rangle
    \;\mbox{for each}\;j\in\mathbb{N}^+.
$$
    This, along with \eqref{yu-5-11-1}, indicates
\begin{equation}\label{yu-5-9-11}
    \|x_{\mathcal{F}}(t_{j\hbar}^+)\|_{\mathbb{R}^n}\leq
    C_{\min}^{-\frac{1}{2}}C_{\max}^{\frac{1}{2}}e^{-\frac{\mu}{2}t_{j\hbar}}
    \|x_{\mathcal{F}}(0)\|_{\mathbb{R}^n}\;\mbox{for
    each}\;j\in\mathbb{N}^+.
\end{equation}
    Arbitrarily fix  $t>t_\hbar$. There is  $j^*\in\mathbb{N}^+$ such that $t_{j^*\hbar}<t\leq
    t_{(j^*+1)\hbar}$. From \eqref{1.5201967}, we have $t_{j^*\hbar}\geq t-t_{\hbar}$. These, together with  (\ref{yu-5-9-11}), yield
\begin{eqnarray*}
    \|x_{\mathcal{F}}(t)\|_{\mathbb{R}^n}&\leq& C\|x_{\mathcal{F}}(t_{j^*\hbar}^+)\|_{\mathbb{R}^n}\leq
    CC_{\min}^{-\frac{1}{2}}C_{\max}^{\frac{1}{2}}e^{-\frac{\mu}{2}t_{j^*\hbar}}
    \|x_{\mathcal{F}}(0)\|_{\mathbb{R}^n}\nonumber\\
    &\leq& CC_{\min}^{-\frac{1}{2}}C_{\max}^{\frac{1}{2}}e^{-\frac{\mu}{2}t_{\hbar}}e^{-\frac{\mu}{2}t}
    \|x_{\mathcal{F}}(0)\|_{\mathbb{R}^n}.
\end{eqnarray*}
   Here $C:=\sup_{s\in[0,t_{\hbar}]}\|S_{\mathcal{F}}(s,0)\|_{\mathcal{L}(\mathbb{R}^n)}$, where $S_{\mathcal{F}}(\cdot,\cdot)$ is the
      transition matrix of the closed-loop system $(A,\{B_kF_k\}_{k=1}^{\hbar},\Lambda_{\hbar})$
      (i.e., (\ref{yu-10-24-4})). So $x_\mathcal{F}(\cdot)$  satisfies  \eqref{stbleexpression}.

     \vskip 5pt

     {\it Step 5. We prove that $(iii)'\Rightarrow (iii)$.}

     Suppose that $(iii)'$ is true. Then by Steps 1-4, we have $(i)$ and $(ii)$.
To show $(iii)$, we arbitrarily fix
$\mathcal{Q}=( Q_j)_{j\in \mathbb{N}^+}\in \mathfrak{M}^{n}_{\hbar,+}$
and $\mathcal{R}=(R_j)_{j\in \mathbb{N}^+}\in \mathfrak{M}^{m}_{\hbar,+}$, and then let $\{\hat P_l\}_{l=0}^\hbar$ be a solution to  (\ref{yu-11-26-1}). It suffices to show
\begin{equation}\label{yu-u-5-10-1}
    V(x_0;l)=\langle \hat P_lx_0,x_0\rangle\;\mbox{for all}\;l\in\{0,1,\ldots,\hbar\}\; \mbox{and}\;x_0\in\mathbb{R}^n,
\end{equation}
 where $V(\cdot;l)$ is given by  (\ref{4.24.2.6}).

   To show
   (\ref{yu-u-5-10-1}), we arbitrarily fix  $l\in \{0,1,\ldots,\hbar\}$ and
   $x_0\in\mathbb{R}^n$, and then arbitrarily fix
       $v=(v_j)_{j\in \mathbb{N}^+}\in\mathcal{U}_{ad}(x_0;l)$. (Notice that $\mathcal{U}_{ad}(x_0;l)\neq \emptyset$, which follows from
    Lemma \ref{yu-lemma-3-12-1} and  $(ii)$ of Theorem \ref{yu-theorem-11-27-1}.)
     Since
     $$
     x(t^+_{j+1};v,x_0,l)=e^{A(t_{j+1}-t_j)}x(t^+_j;v,x_0,l)+B_{j+1}v_{j+1}\;\;\mbox{for all}\;\;
     j\geq l
     $$
     and
$$
t_{j+1}-t_j=t_{\vartheta(j+1)}-t_{\vartheta(j)}\;\;\mbox{for all}\;\;j\geq l,
$$
   we can directly verify from  (\ref{yu-11-26-1})  that when $j\geq l$,
\begin{eqnarray*}
    &\;&\langle \hat P_{\vartheta(j)}x(t^+_{j};v,x_0,l),x(t^+_{j};v,x_0,l)\rangle-\langle
    \hat P_{\vartheta(j+1)}x(t_{j+1}^+;v,x_0,l),x(t_{j+1}^+;v,x_0,l)\rangle\nonumber\\
    &=&[\langle Q_{j+1}x(t_{j+1};v,x_0,j),x(t_{j+1};v,x_0,j)\rangle
    +\langle R_{j+1}v_{j+1},v_{j+1}\rangle]\nonumber\\
        &\;&-\big\|(B_{\vartheta(j+1)}^\top
   \hat P_{\vartheta(j+1)}B_{\vartheta(j+1)}+R_{j+1})^{\frac{1}{2}}v_{j+1}\nonumber\\
        &\;&
    +(B_{\vartheta(j+1)}^\top
    \hat P_{\vartheta(j+1)}B_{\vartheta(j+1)}+ R_{j+1})^{-\frac{1}{2}}B_{\vartheta(j+1)}^\top
    \hat P_{\vartheta(j+1)}e^{A(t_{j+1}-t_j)}x(t^+_{j};v,x_0,l)\big\|^2_{\mathbb{R}^n}.
\end{eqnarray*}
    This, along with the definition of $J(\cdot;x_0,l)$ (see \eqref{yu-11-25-1}), leads to
\begin{eqnarray}\label{2.42201966}
    &\;&J(v;x_0,l)=\langle \hat P_lx_0,x_0\rangle\\
    &\;&+\sum_{j=l+1}^\infty \big\|(B_{\vartheta(j)}^\top
   \hat P_{\vartheta(j)}B_{\vartheta(j)}+R_{j})^{\frac{1}{2}}v_{j}
     +(B_{\vartheta(j)}^\top \hat P_{\vartheta(j)}B_{\vartheta(j)}+
     R_{j})^{-\frac{1}{2}} B_{\vartheta(j)}^\top
     \hat P_{\vartheta(j)}x(t_{j};v,x_0,l)\big\|^2_{\mathbb{R}^n}.\nonumber
\end{eqnarray}
    (The series in the above converges due to $v\in\mathcal{U}_{ad}(x_0;l)$.)
        This, together with
    \eqref{4.24.2.6}, gives that
\begin{equation}\label{yu-6-24-1}
    V(x_0;l)\geq \langle \hat{P}_lx_0,x_0\rangle.
\end{equation}
   Next, we let
    $\mathcal{F}$ be given by \eqref{yu-12-4-1-b} (with $P_k=\hat P_k$). Let $x_\mathcal{F}(\cdot;l)$ be the solution
    to the equation:
\begin{equation*}
\begin{cases}
    x'(t)=Ax(t),&t\in(t_l,+\infty)\setminus\Lambda_{\hbar},\\
    \triangle x(t_j)=B_{\vartheta(j)}F_{\vartheta(j)}x(t_j),&j>l,\\
    x(t_l^+)=x_0.
\end{cases}
\end{equation*}
   Then by taking $\hat{v}=(\hat{v}_j)_{j>l}=(F_{\vartheta(j)}x_{\mathcal{F}}(t_j;l))_{j>l}$,
   we can easily verify  that
$$
    x(t_j;\hat{v},x_0,l)=x(t_j;(F_{\vartheta(k)}x_{\mathcal{F}}(t_k;l))_{k>l},x_0,l)=x_{\mathcal{F}}(t_j;l).
$$
    This, together with (\ref{2.42201966}) and \eqref{4.24.2.6}, yields
$$
   V(x_0;l)\leq  J(\hat{v};x_0,l)=\langle \hat{P}_lx_0,x_0\rangle,
$$
    which, along with \eqref{yu-6-24-1}, leads to
\eqref{yu-u-5-10-1}.

Thus we end the proof of Theorem \ref{yu-theorem-11-27-1}.

\section{Proof of main results (Part II)}

The purpose of this section is to prove Theorem \ref{yu-theorem-3-14-1} and Theorem \ref{yu-theorem-5-13-1}.
\subsection{Preliminary lemmas}
We start with  the controllability of the system $(A,\{B\},\Lambda_{\hbar})$ which is
the system (\ref{yu-10-24-1}) where $B_k=B$ for all $k$.
Given $T>t_1$, we write
\begin{equation*}
     m_T:=\mbox{Card}(\Lambda_{\hbar}\cap (0,T)).
\end{equation*}
   \begin{itemize}
  \item
  The system  $(A,\{B\},\Lambda_{\hbar})$ is said to be controllable at
  time $T>t_1$, if for any $x_0\in\mathbb{R}^n$, there is
  $u:=(u_1,u_2,\cdots,u_{m_T})\odot0\in l^\infty(\mathbb{R}^m)$ so that $x(T;u,x_0)=0$.
\end{itemize}
{\it Recall     (\ref{yu-3-17-4}) and (\ref{yu-3-17-3}) for the definitions of   $q^{n,m}(A,B)$ and $d_A$.}

\begin{lemma}\label{yu-lemma-5-11-2}(\cite[Theorem 2.3.1]{b1}) The system
     $(A,\{B\},\Lambda_{\hbar})$
    is controllable at  $T$ if and only if
$$
    \mbox{Rank}\;(e^{A(T-t_1)}B,e^{A(T-t_2)}B,\cdots, e^{A(T-t_{m_T})}B)=n.
$$
\end{lemma}

\begin{lemma}\label{wang-qin} (\cite[Theorem 2.2]{b2})
    Let  $\{t_j\}_{j=1}^{q^{n,m}(A,B)}\subset\mathbb{R}^+$
    be an increasing strictly sequence satisfying
    $t_{q^{n,m}(A,B)}-t_1<d_A$. Then
\begin{equation*}
    \mbox{Rank}\;(e^{At_1}B,\cdots,e^{At_{q^{n,m}(A,B)}}B)=\mbox{Rank}\;(B,AB,\cdots,A^{n-1}B).
\end{equation*}
    \end{lemma}
    Based on Lemma \ref{yu-lemma-5-11-2} and Lemma \ref{wang-qin}, we can easily obtain the next Lemma \ref{yu-proposition-3-21-1}.
    \begin{lemma}\label{yu-proposition-3-21-1}
    Let $\Lambda_{\hbar}$  verify that $t_{q^{n,m}(A,B)}-t_1<d_A$. Suppose that
\begin{equation*}\label{yu-3-21-b-1-1-b}
  \mbox{Rank}\;(B,AB,\cdots,A^{n-1}B)=n.
\end{equation*}
  Then, for any $T> t_{q^{n,m}(A,B)}$, $(A,\{B\},\Lambda_{\hbar})$ is controllable at $T$.
  \end{lemma}
  \begin{lemma}\label{lemma3.42019610}
  If $(A,\{B\},\Lambda_{\hbar})$ is controllable at some time $T>0$, then it is $\hbar$-stabilizable.
  \end{lemma}
  \begin{proof}
  By the controllability of $(A,\{B\},\Lambda_{\hbar})$ and by \eqref{yu-11-19-2}, we find
  that   $\mathcal{U}_{ad}(x_0)\neq \emptyset$ for each $x_0\in\mathbb{R}^n$. Then
  the $\hbar$-stabilizability of $(A,\{B\},\Lambda_{\hbar})$ follows from
 Theorem \ref{yu-theorem-11-27-1}.
  \end{proof}

    Given $C\in \mathbb{R}^{p\times p}$ and $D\in \mathbb{R}^{p\times q}$ (with $p,q\in \mathbb{N}^+$), we write
     \begin{eqnarray}\label{3.22019617}
     \mathcal{R}[C,D]:=\mbox{Rank}\;(D,CD,\cdots,C^{p-1}D).
      \end{eqnarray}
      \begin{lemma}\label{yu-lemma-5-10-1}
   Suppose that  $\mathcal{R}[A,B]:=r<n$. Then there is an invertible $L\in\mathbb{R}^{n\times
    n}$ so that
\begin{equation}\label{yu-3-17-1}
    LAL^{-1}=\left(
               \begin{array}{cc}
                 A_1 & A_2 \\
                  0& A_3 \\
               \end{array}
             \right),
    \;\;\;\;LB=\left(
                 \begin{array}{c}
                   \widetilde{B} \\
                   0 \\
                 \end{array}
               \right),\;\;\; \mathcal{R}[A_1,\widetilde{B}]=r,
\end{equation}
         where $A_1\in\mathbb{R}^{r\times r}$, $A_2\in\mathbb{R}^{r\times (n-r)}$,
    $A_3\in\mathbb{R}^{(n-r)\times (n-r)}$ and $\widetilde{B}\in\mathbb{R}^{r\times m}$.
    Furthermore, it holds that
    \begin{equation}\label{yu-3-17-5}
\begin{cases}
    q^{n,m}(A,B)=q^{r,m}(A_1,\widetilde{B}),\\
    \sigma(A)=\sigma(A_1)\cup\sigma(A_3).
\end{cases}
\end{equation}
\end{lemma}
   The above (\ref{yu-3-17-1}) is the well-known Kalman controllability decomposition
   (see \cite[Lemma 3.3.3 and Lemma 3.3.4]{b22}), while \eqref{yu-3-17-5}
   can  be directly derived from  (\ref{yu-3-17-1}). (We omit the detailed proof.)
   {\it We call $A_3$ as the uncontrollable part of $(A,B)$.}

   Recall \eqref{yu-3-17-3} and \eqref{yu-3-17-4} for the definitions of $d_A$ and $q^{n,m}(A,B)$.
   Let
   \begin{equation}\label{1.18201969}
    \mathfrak{L}_{A,B}:=\{\{t_j\}_{j\in\mathbb{N}}\subset\mathbb{R}^+:
    \mbox{Card}((s,s+d_A)\cap \{t_j\}_{j\in\mathbb{N}})\geq q^{n,m}(A,B) \;\forall s\in\mathbb{R}^+\}.
\end{equation}

   \begin{lemma}\label{yu-proposition-5-13-1}
    Suppose that  $(iii)$ of  Theorem \ref{yu-theorem-3-14-1},
    where  $\{B_k\}_{k=1}^\hbar=\{B\}$,
    is true. Then for each  $\Lambda_{\hbar}\in\mathfrak{I}_{\hbar}\cap\mathfrak{L}_{A,B}$,
    $(A,\{B\},\Lambda_{\hbar})$ is $\hbar$-stabilizble.
\end{lemma}
               \begin{proof}
               By \eqref{1.18201969} and \eqref{yu-3-18-1}, we see that $\mathfrak{L}_{A,\mathcal{B},\hbar}\subset \mathfrak{L}_{A,B}$, where $\mathcal{B}$ is given by
               \eqref {1.162019611} (with $B_k=B\;\forall\;k$).
               This, along with
               the first note after Theorem \ref{yu-theorem-5-13-1}, we  see that
               $\mathfrak{I}_{\hbar}\cap\mathfrak{L}_{A,B}\neq\emptyset$.
                                   Thus,  we can arbitrarily fix $\Lambda_{\hbar}:=\{t_j\}_{j\in\mathbb{N}}\in\mathfrak{I}_{\hbar}\cap \mathfrak{L}_{A,B}$. This, along with \eqref{1.18201969}, yields
       \begin{eqnarray}\label{3.62019610}
       t_{q^{n,m}(A,B)}-t_1<d_A.
       \end{eqnarray}
            The rest of the proof is organized by two steps.

     \vskip 5pt

    \noindent {\it Step 1. We prove that $(A,\{B\},\Lambda_{\hbar})$ is $\hbar$-stabilizable in the case
     that $\mathcal{R}[A,B]=n$.}

By \eqref{3.62019610} and by the fact that $\mathcal{R}[A,B]=n$, we can apply
Lemma \ref{yu-proposition-3-21-1} to see that $(A,\{B\},\Lambda_{\hbar})$ is controllable at time $t_{q^{n,m}(A,B)+1}$. Then according to Lemma \ref{lemma3.42019610},  $(A,\{B\},\Lambda_{\hbar})$ is $\hbar$-stabilizble.

\vskip 5pt

 \noindent {\it Step 2. We prove that $(A,\{B\},\Lambda_{\hbar})$ is $\hbar$-stabilizable in the case
     that $\mathcal{R}[A,B]=r<n$.}

    First of all, according to Lemma
    \ref{yu-lemma-5-10-1}, there is an invertible matrix
    $L\in\mathbb{R}^{n\times n}$ so that (\ref{yu-3-17-1})-\eqref{yu-3-17-5} hold.
        We now claim
\begin{equation}\label{yu-5-11-3}
   \sigma(A_3)\cap\mathbb{C}^+=\emptyset.
\end{equation}
    If \eqref{yu-5-11-3} was not true, then there would be $\lambda_0\in \sigma(A_3)\cap \mathbb{C}^+$.
       So we have
    $\mbox{Rank}\;(\lambda_0 \mathbb{I}_{n-r}-A_3)<n-r$. This, along with the first two equalities in (\ref{yu-3-17-1}), yields
    \begin{eqnarray}\label{3.8617wan}
    \mbox{Rank}\;(\lambda_0 \mathbb{I}_{n}-A, B)\leq \mbox{Rank}\;(\lambda_0 \mathbb{I}_{r}-A_1, -A_2,  \widetilde{B})+\mbox{Rank}(\lambda_0\mathbb{I}_{n-r}-A_3)<n.
    \end{eqnarray}
    Meanwhile,  by $(iii)$ of  Theorem \ref{yu-theorem-3-14-1}, it follows that $n=\mbox{Rank}\;(\lambda_0 \mathbb{I}_{n}-A, B)$. This contradicts \eqref{3.8617wan}. So \eqref{yu-5-11-3} is true.

   Next, according to \eqref{yu-5-11-3}, there are positive constants $M_1$ and $\mu_1$ so that
    \begin{equation}\label{yu-3-18-4}
    \|e^{A_3t}\|_{\mathcal{L}(\mathbb{R}^{n-r})}\leq M_1e^{-\mu_1 t}\;\mbox{for any}\;t\in\mathbb{R}^+.
\end{equation}
   Meanwhile, from the second equality in (\ref{yu-3-17-5}), (\ref{yu-5-11-3}) and \eqref{yu-3-17-3}, we find that $d_A=d_{A_1}$,
   which, by the first equality in (\ref{yu-3-17-5}), yields  $\mathfrak{L}_{A,B}=\mathfrak{L}_{A_1,\widetilde{B}}$.
    Thus we have $\Lambda_{\hbar}\in \mathfrak{L}_{A_1,\widetilde{B}}$. Because of  this and the last equality in         \eqref{yu-3-17-1}, we can apply
        Lemma \ref{yu-proposition-3-21-1} to get
        the controllability of
    $(A_1,\{\widetilde{B}\},\Lambda_{\hbar})$  at
    $t_{q^{r,m}(A_1,\widetilde{B})+1}=t_{q^{n,m}(A,B)+1}$. From this and Lemma \ref{lemma3.42019610}, we can find
    a feedback law $\widetilde{\mathcal{F}}=\{\widetilde{F}_k\}_{k=1}^\hbar\subset\mathbb{R}^{m\times
    r}$ so that for some positive constants $\mu_2$ and $M_2$,
    \begin{equation}\label{yu-5-9-2}
    \|\widetilde{S}_{\widetilde{\mathcal{F}}}(t,s)\|_{\mathcal{L}(\mathbb{R}^r)}\leq M_2e^{-\mu_2(t-s)}\;\mbox{for any}\;t\geq s\geq 0.
\end{equation}
    Here, $\widetilde{S}_{\widetilde{\mathcal{F}}}(\cdot,\cdot)$  is  the transition matrix of
     the closed-loop system $(A_1, \{\widetilde{B}\widetilde{F}_k\}_{k=1}^\hbar, \Lambda_{\hbar})$
     (i.e.,  \eqref{yu-10-24-4} where $A=A_1$, $B_{k}=\widetilde{B}\;\forall\; k$, $F_k=\widetilde{F}_k\;\forall\; k$).

\par
    We now consider the closed-loop system:
\begin{equation}\label{yu-3-18-6}
\begin{cases}
    y'(t)=A_1y(t)+A_2z(t),&t\in\mathbb{R}^+\setminus \Lambda_{\hbar},\\
    z'(t)=A_3z(t),&t\in\mathbb{R}^+,\\
    \triangle y(t_j)=\widetilde{B}\widetilde{F}_{\vartheta(j)}y(t_j),&j\in\mathbb{N}^+.
\end{cases}
\end{equation}
   Let
   \begin{eqnarray}\label{3.112019610}
   \mathcal{F}:=\{F_k\}_{k=1}^\hbar,\;\;\mbox{with}\;\; F_k:=\left(
           \begin{array}{cc}
             \widetilde{F}_k & 0_{m\times(n-r)} \\
           \end{array}
         \right)L,\;k=1,2,\ldots,\hbar.
\end{eqnarray}
   Two facts are given in order: First, it follows from (\ref{yu-3-18-4}) and (\ref{yu-5-9-2})
   that there are two positive constants $\mu_3$ and $M_3$ so that for each solution $(y(\cdot),z(\cdot))^\top$ to \eqref{yu-3-18-6},
   \begin{equation*}
    \|(y(t),z(t))^\top\|_{\mathbb{R}^n}\leq M_3e^{-\mu_3t} \|(y(0),z(0))^\top\|_{\mathbb{R}^n}\;\;\mbox{for each}\;\;t\geq 0.
\end{equation*}
   (Here, we used that $ y(t)=\widetilde{S}_{\widetilde{\mathcal{F}}}(t,0)y(0)+\int_0^t
    \widetilde{S}_{\widetilde{\mathcal{F}}}(t,s)A_2z(s)ds$ and $z(t)=e^{A_3t}z(0)$.)
   Second, $(y(\cdot),z(\cdot))^\top$ solves \eqref{yu-3-18-6} if and only if
   $x(\cdot):=L^{-1}(y(\cdot),z(\cdot))^\top$ solves \eqref{yu-10-24-4}
   where $B_k=B\;\forall\;k$ and  $\mathcal{F}$ is given by \eqref{3.112019610}.

   Finally, from the above two facts, we see that $(A,\{B\},\Lambda_{\hbar})$ is $\hbar$-stabilizable. This ends the proof of  Lemma \ref{yu-proposition-5-13-1}.
   \end{proof}

   \subsection{Key  proposition }

 Recall \eqref{1.162019611} and \eqref{yu-3-18-1} for the definitions of $\mathcal{B}$
   and $\mathfrak{L}_{A,\mathcal{B},\hbar}$.
\begin{proposition}\label{yu-proposition-5-13-2}
    Suppose that $(iii)$ of Theorem \ref{yu-theorem-3-14-1}
    is true. Then for any $\Lambda_{\hbar}\in\mathfrak{I}_{\hbar}\cap\mathfrak{L}_{A,\mathcal{B},\hbar}$, $(A,\{B_k\}_{k=1}^{\hbar},\Lambda_{\hbar})$ is $\hbar$-stabilizable.
\end{proposition}

\begin{proof}
Recall that for each $\hbar\in \mathbb{N}^+$,
$\mathfrak{I}_{\hbar}\cap\mathfrak{L}_{A,\mathcal{B},\hbar}\neq\emptyset$ (see the first note after Theorem \ref{yu-theorem-5-13-1}).
 The proof is organized by three parts: $\hbar=1$; $\hbar=2$;  $\hbar\geq 3$.
 \vskip 5pt

 \noindent
 {\it Part 1: We prove Proposition \ref {yu-proposition-5-13-2} for the case that $\hbar=1$.}

  In this case,  Proposition \ref{yu-proposition-5-13-2} follows from Lemma \ref{yu-proposition-5-13-1}
  since $\mathfrak{L}_{A,\mathcal{B},1}\subset \mathfrak{L}_{A,B}$ (with $\mathcal{B}=B$), where $\mathfrak{L}_{A,B}$ is defined by (\ref{1.18201969}).

\vskip 5pt

 \noindent
 {\it Part 2:  We prove Proposition \ref {yu-proposition-5-13-2} for the case that $\hbar=2$.}

 In this case, we have $\mathcal{B}=\left(
    \begin{array}{cccc}
    B_1 & B_2  \\
    \end{array}
    \right)$.
 Suppose that $(iii)$ of Theorem \ref{yu-theorem-3-14-1}
    is true.
            Arbitrarily fix $\Lambda_{2}:=\{t_j\}_{j\in\mathbb{N}}$
            so that
            \begin{eqnarray}\label{3.112019617}
                 \Lambda_2\in \mathfrak{I}_{2}\;\;\mbox{and}\;\;\Lambda_2\in\mathfrak{L}_{A,\mathcal{B},2}.                 \end{eqnarray}
                   {\it We aim to show the $2$-stabilizability of $(A,\{B_k\}_{k=1}^{2},\Lambda_{2})$.  }
                 Let $\Lambda:=\{t_{2j}\}_{j\in\mathbb{N}}$. By \eqref{instantset}, we have
                 \begin{eqnarray}\label{3.122019617}
                 \Lambda\in \mathfrak{I}_1.
                 \end{eqnarray}

                       We first consider the case that
           \begin{eqnarray}\label{3.156-18-1}
           \mathcal{R}[A,B_2]=n.
           \end{eqnarray}
              Two observations are given in order: First, from \cite[Lemma 3.3.7]{b22} and \eqref{3.22019617},
            we have
            \begin{eqnarray}\label{3.142019617}
            \mathcal{R}[A,B_2]=n \Rightarrow
             (iii) \;\mbox{of Theorem \ref{yu-theorem-3-14-1}  (with}\;\{B_k\}_{k=1}^\hbar=\{B_2\}\;\mbox{and any}\; \hbar\in \mathbb{N}^+).
            \end{eqnarray}
            Second, by \eqref{yu-3-17-4}, we have     $q^{n,m}(A,B_{2})\leq q^{n,2m}(A,\mathcal{B})$. This, along with the second equality in \eqref{3.112019617} and
              \eqref{yu-3-18-1}, shows that  for any $s\in\mathbb{R}^+$,
$$
    2(\mbox{Card}((s,s+d_A)\cap \Lambda)+2)\geq \mbox{Card}((s,s+d_A)\cap \Lambda_{2})\geq
    2(q^{n,2m}(A,\mathcal{B})+2)\geq 2(q^{n,m}(A,B_{2})+2),
$$
    which, along with \eqref{1.18201969} and \eqref{3.122019617}, leads to
    \begin{eqnarray}\label{3.152019617}
    \Lambda\in\mathfrak{L}_{A,B_{2}}\cap \mathfrak{I}_1,
    \end{eqnarray}
     where $\mathfrak{L}_{A,B_2}$ is given by (\ref{1.18201969}) (with $B=B_2$).   From \eqref{3.156-18-1} \eqref{3.142019617} and \eqref{3.152019617},
        we can apply  Lemma \ref{yu-proposition-5-13-1} (where $(\hbar,A,\{B\},\Lambda_\hbar)$ is replaced by
        $(1,A,\{B_2\},\Lambda)$)
                   to find  $\widehat F_{2}\in \mathbb{R}^{m\times n}$ so that the closed-loop system $(A,\{B_2 \widehat F_{2}\},\Lambda)$ (i.e.,   \eqref{yu-10-24-4} where $\hbar=1$, $\Lambda_\hbar=\Lambda$,
                   $B_{\vartheta(j)}=B_2\;\forall\;j$, $F_{\vartheta(j)}=\widehat F_2\;\forall\;j$)
                    is stable.
    From this, we can easily see that  the closed-loop system
    $(A,\{B_k\widehat F_k\}_{k=1}^{2},\Lambda_{2})$, with $\widehat F_1:=0$, (see the corresponding \eqref{yu-10-24-4} with $F_1=\widehat{F}_1$ and $F_2=\widehat{F}_2$)
     is stable. Hence
    $(A,\{B_k\}_{k=1}^{2},\Lambda_{2})$ is $2$-stabilizable.
\par

   Next, we turn to the main part of the proof:
\vskip 5pt
    {\it When $\mathcal{R}[A,B_2]:=n_2<n$,  $(A,\{B_k\}_{k=1}^{2},\Lambda_{2})$ is $2$-stabilizable.}
\vskip 5pt
    \noindent This will be carried by several steps.

   \vskip 5pt

   \noindent {\it Step 1. We give a decomposition and a related decay estimate.}

      Since $\mathcal{R}[A,B_2]:=n_2<n$, we can use Lemma \ref{yu-lemma-5-10-1} to find
      an invertible $L_2\in\mathbb{R}^{n\times m}$
    so that
\begin{equation}\label{yu-5-1-4-b}
    L_2 AL_2^{-1}=\left(
                 \begin{array}{cc}
                   A_{2,1} & A_{2,2} \\
                   0 & A_{2,3} \\
                 \end{array}
               \right),\;\;
    L_2 B_2=\left(
             \begin{array}{c}
               \widetilde{B}_{2} \\
               0 \\
             \end{array}
           \right), \; \mathcal{R}[A_{2,1},\widetilde{B}_{2}]=n_{2}
\end{equation}
     where $A_{2,1}\in\mathbb{R}^{n_{2}\times n_{2}}$, $A_{2,2}\in\mathbb{R}^{n_{2}\times(n-n_{2})}$,
    $A_{2,3}\in\mathbb{R}^{(n-n_2)\times(n-n_{2})}$ and $\widetilde{B}_2\in\mathbb{R}^{n_2\times m}$.
     Let
\begin{equation}\label{yu-5-4-b-1-b}
    L_{2}B_{1}:=\left(
                            \begin{array}{c}
                              \widehat{B}_{1} \\
                              B_{1,2} \\
                            \end{array}
                          \right),\;\;\mbox{with}\; \widehat{B}_{1}\in\mathbb{R}^{n_{2}\times m},
                          B_{1,2}\in\mathbb{R}^{(n-n_{2})\times m}.
\end{equation}

    We now claim
    \begin{eqnarray}\label{3.142019612-1}
    \Lambda\in \mathfrak{L}_{A_{2,1},\widetilde{B}_{2}}\;\;\mbox{and}\;\;\Lambda\in \mathfrak{I}_1.
    \end{eqnarray}
    Indeed,  the second conclusion in \eqref{3.142019612-1}
    follows from \eqref{instantset} directly (since $\Lambda:=\{t_{2j}\}_{j\in\mathbb{N}}$).
    Meanwhile,
       one can directly verify from  \eqref{yu-3-18-1} and \eqref{3.112019617}
     that for any $s\in\mathbb{R}^+$,
$$
    2(\mbox{Card}((s,s+d_{A_{2,1}})\cap\Lambda)+2)\geq \mbox{Card}((s,s+d_A)\cap \Lambda_{2})
    \geq 2(q^{n,2m}(A,\mathcal{B})+2)\geq 2(q^{n_{2},m}(A_{2,1},\widetilde{B}_{2})+2).
$$
    Here, we note that $d_{A_{2,1}}\geq d_A$ by the second equality in (\ref{yu-3-17-5}). This implies that
\begin{equation}\label{yu-5-8-1}
    \mbox{Card}((s,s+d_{A_{2,1}})\cap\Lambda)\geq q^{n_{2},m}(A_{2,1},\widetilde{B}_{2})\;\forall\; s\in\mathbb{R}^+.
\end{equation}
    From \eqref{1.18201969} and \eqref{yu-5-8-1}, we obtain the first conclusion in \eqref{3.142019612-1}.

Next, from the last equality in \eqref{yu-5-1-4-b}, we can use \eqref{3.142019617}
 (where $(A,B_2,n)$ is replaced by $(A_{2,1},\widetilde{B}_{2},n_2)$)
  to get $(iii)$ of Theorem \ref{yu-theorem-3-14-1} (where $(\hbar,A,\{B_k\}_{k=1}^\hbar)$ is replaced by
  $(1,A_{2,1},\{\widetilde{B}_{2}\})$).
   From this and  \eqref{3.142019612-1}, we can use
     Lemma \ref{yu-proposition-5-13-1} (where $(\hbar,A,\{B\},\Lambda_\hbar)$ is replaced by
     $(1,A_{2,1},\{\widetilde{B}_{2}\},\Lambda)$)
              to find $\widetilde{F}_{2}\in\mathbb{R}^{m\times n_{2}}$ so that  the closed-loop system  $(A_{2,1},\{\widetilde{B}_{2}\widetilde{F}_{2}\},\Lambda)$ (see the corresponding  \eqref{yu-10-24-4} with $\hbar=1$, $A=A_{2,1}$,
              $B_1=\widetilde{B}_2$, $F_1=\widetilde{F}_2$)
                             is stable.
    Thus, there is $\mu_{2}>0$ and $M_{2}>0$ so that
    \begin{equation}\label{yu-5-2-1-b}
    \|\widetilde{S}_{2}(t,s)\|_{\mathcal{L}(\mathbb{R}^{n-n_2})}\leq M_{2}e^{-\mu_{2}(t-s)}
    \;\mbox{for any}\;\;t\geq s\geq 0,
\end{equation}
     where $\widetilde{S}_{2}(\cdot,\cdot)$ is the transition matrix
     generated by the closed-loop system $(A_{2,1},\{\widetilde{B}_{2}\widetilde{F}_{2}\},\Lambda)$.

    \vskip 5pt

    \noindent{\it Step 2. With notations in \eqref{yu-5-1-4-b} and \eqref{yu-5-4-b-1-b}, we prove that when $\sigma(A_{2,3})\cap\mathbb{C}^+=\emptyset$,
    $(A,\{B_k\}_{k=1}^{2},\Lambda_{2})$ is $2$-stabilizable.}

Since $\sigma(A_{2,3})\cap\mathbb{C}^+=\emptyset$, there is  $\mu'_2>0$ and $M_{2}'>0$ so that
\begin{equation}\label{yu-5-1-2-b}
    \|e^{A_{2,3} t}\|_{\mathcal{L}(\mathbb{R}^{n-n_2})}\leq M_{2}' e^{-\mu'_2 t}\;\mbox{for any}\;t\in\mathbb{R}^+.
\end{equation}
    We consider the following closed-loop system (with $\widetilde{F}_{2}$ given in Step 1):
\begin{equation}\label{yu-5-2-2-b}
\begin{cases}
    y_{2}'(t)=A_{2,1}y_{2}(t)+A_{2,2}z_{2}(t),&t\in\mathbb{R}^+\setminus\Lambda_{2},\\
    z_{2}'(t)=A_{2,3}z_{2}(t),&t\in\mathbb{R}^+,\\
    \triangle y_{2}(t_j)=\widetilde{B}_{2}\widetilde{F}_{2}y_{2}(t_j),&\mbox{if}\;\vartheta(j)=2,\\
    \triangle y_{2}(t_j)=0,&\mbox{if}\;\vartheta(j)\neq 2.
\end{cases}
\end{equation}
       Let
\begin{eqnarray}\label{3.182019612}
   \mathcal{F}:=\{F_k\}_{k=1}^2\;\;\mbox{with}\; F_1:=0,\;\; F_{2}:=\left(
                \begin{array}{cc}
                  \widetilde{F}_{2} & 0_{m\times(n-n_{2})} \\
                \end{array}
              \right)L_{2}.
\end{eqnarray}
Two facts are given in order: First, by \eqref{yu-5-2-2-b}, (\ref{yu-5-2-1-b}) and (\ref{yu-5-1-2-b}),
   there is $\mu''_2>0$ and $M_2''>0$ so that each solution
   $(y_2(\cdot),z_2(\cdot))^\top$ (to \eqref{yu-5-2-2-b}) satisfies
   \begin{eqnarray*}
   \|(y_2(t),z_2(t))^\top\|_{\mathbb{R}^n}\leq M''_{2}e^{-\mu''_2 t}\|(y_2(0),z_2(0))^\top\|_{\mathbb{R}^n}\;\forall\; t\geq 0.
   \end{eqnarray*}
   Second, $(y_2(\cdot),z_2(\cdot))^\top$ solves \eqref{yu-5-2-2-b} if and only if
   $x(\cdot):=L^{-1}_{2}(y_2(\cdot),z_2(\cdot))^\top$ solves
   \eqref{yu-10-24-4} where $\hbar=2$, $\Lambda_{\hbar}=\Lambda_{2}$ and $\mathcal{F}$ is
   given by \eqref{3.182019612}. From these two facts, we see that
    $(A,\{B_k\}_{k=1}^2,\Lambda_{2})$ is $2$-stabilizable.

  \vskip 5pt

    \noindent{\it Step 3.  With notations in \eqref{yu-5-1-4-b} and \eqref{yu-5-4-b-1-b},
    we prove that $(A,\{B_k\}_{k=1}^{2},\Lambda_{2})$ is $2$-stabilizable,     when $\sigma(A_{2,3})\cap\mathbb{C}^+\neq\emptyset$ and $\mathcal{R}[A_{2,3},B_{1,2}]=n-n_{2}$.}

   We can use a very similar way used in the proof of \eqref{3.142019612-1} to  show
     $\Lambda\in \mathfrak{L}_{A_{2,3},B_{1,2}}\cap\mathfrak{I}_1$ for this case.
     Meanwhile, since $\mathcal{R}[A_{2,3},B_{1,2}]=n-n_{2}$,
     we can use \eqref{3.142019617} (where $(A,B_2,n)$ is replaced by $(A_{2,3},B_{1,2},n-n_2)$) to get  $(iii)$ of  Theorem \ref{yu-theorem-3-14-1} (where $(\hbar,A,\{B_k\}_{k=1}^\hbar)$ is replaced by $(1,A_{2,3},\{B_{1,2}\})$).
          From these,  we can apply
    Lemma \ref{yu-proposition-5-13-1}
    (where $(\hbar,A,\{B\},\Lambda_\hbar)$ is replaced by $(1,A_{2,3},\{B_{1,2}\},\Lambda)$)
                      to find $F_{1,2}\in\mathbb{R}^{m\times(n-n_{2})}$ so that
     for some $\mu_{1,2}>0$ and $M_{1,2}>0$,
      each solution ${w}_{1,2}(\cdot)$ to the closed-loop system
       $(A_{2,3},\{B_{1,2}F_{1,2}\},\Lambda)$
      (see the corresponding  \eqref{yu-10-24-4}
      with $\hbar=1$, $A=A_{2,3}$, $B_1=B_{1,2}$,
      $F_1=F_{1,2}$) satisfies
\begin{equation}\label{yu-5-4-1-b}
    \|{w}_{1,2}(t)\|_{\mathbb{R}^{(n-n_2)}}
    \leq M_{1,2}e^{-\mu_{1,2}(t-s)}
    \|{w}_{1,2}(s)\|_{\mathbb{R}^{(n-n_2)}}
    \;\;\mbox{for any}\;\;t\geq s\geq0.
\end{equation}
          Next, let $\widetilde{S}_{1,2}(\cdot,\cdot)$
    be the transition matrix of the system:
\begin{equation}\label{3.222019613-11}
\begin{cases}
    \widetilde{w}'_{1,2}(\sigma)=A_{2,3}\widetilde{w}_{1,2}(\sigma), &\sigma\in\mathbb{R}^+\setminus \{t_{2j-1}\}_{j\in\mathbb{N}^+},\\
    \triangle \widetilde{w}_{1,2}(t_{2j-1})=B_{1,2}F_{1,2}
    \widetilde{w}_{1,2}(t_{2j-1}),&j\in\mathbb{N}^+.
\end{cases}
\end{equation}
One can easily check that  $w_{1,2}(\cdot)$ solves the closed-loop system
       $(A_{2,3},\{B_{1,2}F_{1,2}\},\Lambda)$ if and only if
 $w_{1,2}(\cdot+t_2-t_1)$ solves \eqref{3.222019613-11}. This, along with  (\ref{yu-5-4-1-b}), leads to
    \begin{equation}\label{yu-5-4-2-b}
    \|\widetilde{S}_{1,2}(\sigma,0)\|_{\mathcal{L}(\mathbb{R}^{n-n_2})}\leq M_{1,2}e^{-\mu_{1,2}\sigma}\;\;\mbox{for any}\;\;\sigma\in\mathbb{R}^+.
\end{equation}

We now consider the closed-loop system:
\begin{equation}\label{yu-5-4-3-b}
\begin{cases}
    \left(
             \begin{array}{c}
               y_{2}(t) \\
               z_{2}(t) \\
             \end{array}
           \right)'=\left(
                      \begin{array}{cc}
                        A_{2,1} & A_{2,2} \\
                        0 & A_{2,3} \\
                      \end{array}
                    \right)\left(
             \begin{array}{c}
               y_{2}(t) \\
               z_{2}(t) \\
             \end{array}
           \right),&t\in\mathbb{R}^+\setminus\Lambda_{2}, \\
    \triangle y_{2}(t_j)=\widetilde{B}_{2}\widetilde{F}_{2}y_{2}(t_j),&\mbox{if}\;
    \vartheta(j)=2,\\
    \triangle z_{2}(t_j)=B_{1,2}F_{1,2}z_{2}(t_j),
    &\mbox{if}\;\vartheta(j)=1,\\
    \triangle y_{2}(t_j)=0,&\mbox{if}\;\vartheta(j)\neq 2, \\
    \triangle z_{2}(t_j)=0,&\mbox{if}\;\vartheta(j)\neq 1.
\end{cases}
\end{equation}
    Let
\begin{eqnarray}\label{3.242019612-11}
    \mathcal{F}:=\{F_k\}_{k=1}^2,\; \mbox{with}\;
    F_{1}:=\left(
                   \begin{array}{cc}
                     0_{m\times n_{2}} & F_{1,2} \\
                   \end{array}
                 \right)L_{2},
    \;\;F_{2}:=\left(
                \begin{array}{cc}
                  \widetilde{F}_{2} & 0_{m\times(n-n_{2})} \\
                \end{array}
              \right)L_{2}.
\end{eqnarray}
Several facts are given in order: First, since  $\vartheta(2j-1)=1\;\forall\; j\in\mathbb{N}^+$ (with $\hbar=2$), the solution $(y_2(\cdot), z_2(\cdot))^\top$ to the equation (\ref{yu-5-4-3-b}) can be expressed by
\begin{equation*}
    (y_{2}(t), z_{2}(t))^\top=\Big(\widetilde{S}_{2}(t,0)y_{2}(0)+\int_0^t\widetilde{S}_{2}(t,s)A_{2,2}
    z_{2}(s)ds,\;
   \widetilde{S}_{1,2}(t,0)z_{2}(0)\Big)^\top\;\forall\; t\geq 0.
\end{equation*}
  Second, it follows by the first fact, (\ref{yu-5-2-1-b}) and (\ref{yu-5-4-2-b}) that for some  $\mu_{1,2}'>0$  and $M'_{1,2}>0$, any solution $(y_{2}(\cdot), z_{2}(\cdot))^\top$ to (\ref{yu-5-4-3-b})
  satisfies
\begin{eqnarray*}
\|(y_{2}(t), z_{2}(t))^\top\|_{\mathbb{R}^n}\leq M'_{1,2}e^{-\mu_{1,2}' t}\|(y_{2}(0), z_{2}(0))^\top\|_{\mathbb{R}^n}\;\;\forall\; t\geq 0,
\end{eqnarray*}
i.e., the closed-loop system (\ref{yu-5-4-3-b}) is stable.
  Third, $(y_2(\cdot), z_2(\cdot))^\top$ solves (\ref{yu-5-4-3-b}) if and only if
  $x(\cdot):=L_{2}^{-1}(y_2(\cdot), z_2(\cdot))^\top$ solves
   \eqref{yu-10-24-4}
   where $\hbar=2$,  $\Lambda_{\hbar}=\Lambda_{2}$ and $\mathcal{F}$ is given by
  \eqref{3.242019612-11}.

  Finally, the last two facts above leads to the $2$-stabilization of $(A,\{B_{k}\}_{k=1}^{2},\Lambda_{2})$.

\vskip 5pt

    \noindent{\it Step 4. With notations in \eqref{yu-5-1-4-b} and \eqref{yu-5-4-b-1-b},
    we prove that $(A,\{B_k\}_{k=1}^{2},\Lambda_{2})$ is $2$-stabilizable,     when $\sigma(A_{2,3})\cap\mathbb{C}^+\neq\emptyset$ and $\mathcal{R}[A_{2,3},B_{1,2}]:=n_1<n-n_{2}$.}

The proof of this step is divided into several sub-steps.

\vskip 5pt

\noindent{\it Sub-step 4.1. We give another decomposition and a related decay estimate.}

   Since  $n_1<n-n_{2}$, we can use Lemma \ref{yu-lemma-5-10-1} to find  an
    invertible  $\widetilde{L}_{1}\in\mathbb{R}^{(n-n_{2})\times (n-n_{2})}$ so that
\begin{eqnarray}\label{3.262019613}
    \widetilde{L}_{1}A_{2,3}\widetilde{L}^{-1}_{1}
    =\left(                                                                       \begin{array}{cc}                                                                                 A_{1,1} & A_{1,2} \\                                                                                 0 & A_{1,3} \\                                                                               \end{array}                                                                             \right),\;\;
    \widetilde{L}_{1}{B}_{1,2}
    =\left(
    \begin{array}{c}
    \widetilde{B}_{1} \\
     0 \\
    \end{array}
     \right),\;\;
      \mathcal{R}[A_{1,1},\widetilde{B}_{1}]=n_{1},                                                      \end{eqnarray}
    where $A_{1,1}\in\mathbb{R}^{n_{1}\times n_{1}}$,
    $A_{1,2}\in\mathbb{R}^{n_{1}\times (n-n_1-n_{2})}$, $A_{1,3}\in\mathbb{R}^{(n-n_{1}-n_{2})\times (n-n_{1}-n_{2})}$
    and $\widetilde{B}_{1}\in\mathbb{R}^{n_{1}\times m}$. Let
\begin{eqnarray}\label{3.262019614}
    L_{1}:=\left(
                  \begin{array}{cc}
                    \mathbb{I}_{n_{2}} & 0 \\
                    0 & \widetilde{L}_{1} \\
                  \end{array}
                \right).
\end{eqnarray}
    By \eqref{3.262019613}, \eqref{3.262019614}, (\ref{yu-5-1-4-b}) and (\ref{yu-5-4-b-1-b}), there is  $\widehat{A}_{2,1}\in\mathbb{R}^{n_{2}\times n_1}$ and
    $\widehat{A}_{2,2}\in \mathbb{R}^{n_{2}\times (n-n_{1}-n_{2})}$  so that
\begin{equation}\label{yu-5-5-7-b}
    L_{1}L_{2}A(L_{1}L_{2})^{-1}=
    \left(
      \begin{array}{ccc}
        A_{2,1} & \widehat{A}_{2,1} & \widehat{A}_{2,2} \\
        0 & A_{1,1} & A_{1,2} \\
        0 & 0 & A_{1,3} \\
      \end{array}
    \right),\;\;
    L_{1}L_{2}\left(
                          \begin{array}{cc}
                            B_{1} & B_{2} \\
                          \end{array}
                        \right)=\left(
                                  \begin{array}{cc}
                                    \widehat{B}_{1} & \widetilde{B}_{2} \\
                                    \widetilde{B}_{1} & 0 \\
                                    0 & 0 \\
                                  \end{array}
                                \right).
\end{equation}

   By a very similar way used in the proof of \eqref{3.142019612-1}, we can show
   $\Lambda\in \mathfrak{L}_{A_{1,1},\widetilde{B}_{1}}\cap \mathfrak{I}_1$.
     Meanwhile, by the last equality in \eqref{3.262019613},
     we can use \eqref{3.142019617} (where $(A,B_2,n)$ is replaced by $(A_{1,1}, \widetilde{B}_{1}, n_1)$)
     to get  $(iii)$ of  Theorem \ref{yu-theorem-3-14-1} (where $(\hbar,A,\{B_k\}_{k=1}^\hbar)$ is replaced by
     $(1,A_{1,1},\{\widetilde{B}_{1}\})$).
               From these, we can use
   Lemma \ref{yu-proposition-5-13-1} (where $(\hbar,A,\{B\},\Lambda_\hbar)$ is replaced by
   $(1,A_{1,1},\{\widetilde{B}_{1}\},\Lambda)$)
   to find $\widetilde{F}_{1}\in \mathbb{R}^{m\times n_{1}}$ so that for some $\mu_1>0$ and $M_1>0$,
   each solution $w_1(\cdot)$ to
   the closed-loop system $(A_{1,1},\{\widetilde{B}_{1} \widetilde{F}_{1}\},\Lambda)$ (see the corresponding  \eqref{yu-10-24-4}
    with $\hbar=1$, $A=A_{11}$, $B_1=\widetilde{B}_1$, $F_1=\widetilde{F}_1$)
       satisfies
   \begin{equation}\label{yu-5-5-2-b}
    \|w_{1}(t)\|_{\mathbb{R}^{n_1}}\leq M_{1}e^{-\mu_{1}(t-s)}\|w_{1}(s)\|_{\mathbb{R}^{n_1}}\;\;\mbox{for any}\;\;t\geq s\geq 0.
\end{equation}
    Next, we let $\widetilde{S}_{1}(\cdot,\cdot)$ be the transition matrix of the system:
    \begin{equation}\label{3.312019613-12}
\begin{cases}
    \widetilde{w}_{1}'(\sigma)=A_{1,1}\widetilde{w}_{1}(\sigma),
    &\sigma\in\mathbb{R}^+\setminus\{t_{2j-1}\}_{j\in\mathbb{N}^+},\\
    \triangle\widetilde{w}_{1}(t_{2j-1})=
    \widetilde{B}_{1}\widetilde{F}_{1}\widetilde{w}_{1}(t_{2j-1}),&j\in\mathbb{N}^+.
\end{cases}
\end{equation}
One can easily check that  $w_{1}(\cdot)$ solves the closed-loop system
       $(A_{1,1},\{\widetilde{B}_{1} \widetilde{F}_{1}\},\Lambda)$ if and only if
 $w_{1}(\cdot+t_2-t_1)$ solves \eqref{3.312019613-12}. This, along with
 (\ref{yu-5-5-2-b}), yields
 \begin{equation}\label{yu-5-5-1-b}
    \|\widetilde{S}_{1}(t,s)\|_{\mathcal{L}(\mathbb{R}^{n_1})}\leq M_{1}e^{-\mu_{1}(t-s)}\;\mbox{for any}\;t\geq s\geq 0.
\end{equation}
\vskip 5pt

\noindent{\it Sub-step 4.2. We prove
   \begin{equation}\label{yu-5-8-2}
    \sigma(A_{1,3})\cap \mathbb{C}^+=\emptyset.
\end{equation}}

     By contradiction, suppose that \eqref{yu-5-8-2} were not true. Then there would be
a number $\lambda_0$ so that
\begin{eqnarray}\label{3.342019614}
\lambda_0\in \sigma(A_{1,3})\cap\mathbb{C}^+\;\;\mbox{and}\;\;\mbox{Rank}
    (\lambda_0\mathbb{I}_{n-n_1-n_2}-A_{1,3})<n-n_1-n_2.
\end{eqnarray}
Next, with notations in \eqref{yu-5-1-4-b}, \eqref{yu-5-4-b-1-b}, \eqref{3.262019613} and
\eqref{yu-5-5-7-b}, we let
\begin{eqnarray}\label{3.342-19614}
    \mathcal{A}:=\left(
                   \begin{array}{cc}
                     A_{2,1} & \widehat{A}_{2,1} \\
                     0 & A_{1,1} \\
                   \end{array}
                 \right),\;\;
    \widehat{\mathcal{A}}:=\left(
                   \begin{array}{c}
                     \widehat{A}_{2,2} \\
                     A_{1,2} \\
                   \end{array}
                 \right),\;\;\widetilde{\mathcal{B}}:
                 =\left(
                \begin{array}{cc}
                \widehat{B}_1 & \widetilde{B}_2 \\
                \widetilde{B}_1 & 0 \\
                 \end{array}
                      \right).
\end{eqnarray}
From \eqref{yu-5-5-7-b} and \eqref{3.342-19614}, we can easily check that
     \begin{eqnarray}\label{yu-5-8-3}
    \mbox{Rank}\;(\lambda_0\mathbb{I}_{n}-A,\mathcal{B})
    &=&\mbox{Rank}\;(\lambda_0\mathbb{I}_n-
    L_1L_2A
    (L_1L_2)^{-1}, L_1L_2\left(
      \begin{array}{cc}
     B_1&B_2  \\
       \end{array}
        \right))\nonumber\\
    &=&\mbox{Rank}\left(\left(
     \begin{array}{cc}
     \lambda_0\mathbb{I}_{n_1+n_2}-\mathcal{A} & -\widehat{\mathcal{A}}\\                          0&\lambda_0\mathbb{I}_{n-n_1-n_2}-A_{1,3}\\
      \end{array}
       \right),\left(
       \begin{array}{cc}
        \widetilde{\mathcal{B}} \\
         0\\
          \end{array}
          \right)\right).
\end{eqnarray}
    Since $ \mbox{Rank}\;((\lambda_0\mathbb{I}_{n_1+n_2}-\mathcal{A},-\widehat{\mathcal{A}}),
    \widetilde{\mathcal{B}})\leq n_1+n_2$, it follows from (\ref{yu-5-8-3}) and the second inequality in
    \eqref{3.342019614} that
\begin{eqnarray*}
    \mbox{Rank}\;(\lambda_0\mathbb{I}_n-A,\mathcal{B})\leq \mbox{Rank}\;((\lambda_0\mathbb{I}_{n_1+n_2}-\mathcal{A},-\widehat{\mathcal{A}}),
    \widetilde{\mathcal{B}})+\mbox{Rank}
    (\lambda_0\mathbb{I}_{n-n_1-n_2}-\widetilde{A}_1)<n.
\end{eqnarray*}
   This contradicts  $(iii)$ of Theorem \ref{yu-theorem-3-14-1} (which is the assumption of Proposition \ref{yu-proposition-5-13-2}).
       So (\ref{yu-5-8-2}) is true.

   \vskip 5pt

\noindent{\it Sub-step 4.3. We finish the proof of Step 4.}

\par
    By (\ref{yu-5-8-2}), there is  $\mu'_{1}>0$ and $M_{1}'>0$ so that
\begin{equation}\label{yu-5-4-5-b}
    \|e^{A_{1,3}t}\|_{\mathcal{L}(\mathbb{R}^{n-n_1-n_2})}\leq M_{1}'e^{-\mu'_{1}t}\;\;\mbox{for any}\;\;t\in\mathbb{R}^+.
\end{equation}
  With $\widetilde{F}_1$ and $\widetilde{F}_2$ given by Sub-step 4.1 and Step 1,  we consider the closed-loop system:
\begin{equation}\label{yu-5-5-5-b}
\begin{cases}
    \left(
      \begin{array}{c}
        y_{2}(t) \\
        y_{1}(t) \\
        z_{1}(t) \\
      \end{array}
    \right)'=\left(
    \begin{array}{ccc}
    A_{2,1} & \widehat{A}_{2,1} & \widehat{A}_{2,2} \\
    0 & A_{1,1} & A_{1,2} \\
    0 & 0 & A_{1,3} \\
     \end{array}
    \right)\left(
      \begin{array}{c}
        y_{2}(t) \\
        y_{1}(t) \\
        z_{1}(t) \\
      \end{array}
    \right),&t\in\mathbb{R}^+\setminus\Lambda_{2},\\
    \triangle y_{2}(t_j)=\widetilde{B}_{2}\widetilde{F}_{2}y_{2}(t_j),&\mbox{if}\;\vartheta(j)=2,\\
    \triangle y_{1}(t_j)=\widetilde{B}_{1}\widetilde{F}_{1} y_{1}(t_j),
    &\mbox{if}\;\vartheta(j)=1,\\
    \triangle y_{2}(t_j)=0,&\mbox{if}\;\vartheta(j)\neq 2,\\
    \triangle y_{1}(t_j)=0,&\mbox{if}\;\vartheta(j)\neq 1,\\
    \triangle z_{1}(t_j)=0,&j\in\mathbb{N}^+
\end{cases}
\end{equation}
     and let
   \begin{eqnarray}\label{3.402019614}
       F_{1}:=\left(
    \begin{array}{ccc}
    0_{m\times n_{2}} & \widetilde{F}_{1} & 0_{m\times(n-n_{1}-n_{2})} \\
    \end{array}
    \right)L_{1}L_{2},\;
    F_{2}:=\left(
    \begin{array}{cc}
    \widetilde{F}_{2} & 0_{m\times(n-n_{2})} \\
    \end{array}
    \right)L_{1}L_{2},
\end{eqnarray}
   Several facts are given in order: First,    (\ref{yu-5-5-5-b}) is equivalent to
\begin{equation*}\label{yu-5-5-8-b}
\begin{cases}
    y_{2}(t)=\widetilde{S}_{2}(t,0)y_{2}(0)+\int_0^t\widetilde{S}_{2}(t,s)
    \left(\widehat{A}_{2,1}y_{1}(s)ds+\widehat{A}_{2,2}z_{1}(s)\right)ds,\\
    y_{1}(t)=\widetilde{S}_{1}(t,0)y_{1}(0)
    +\int_0^t\widetilde{S}_{1}(t,s)A_{1,2}z_{1}(s)ds,\\
    z_{1}(t)=e^{A_{1,3}t}z_{1}(0).
\end{cases}
    \mbox{for any}\;t\in\mathbb{R}^+.
\end{equation*}
  Second, from the first fact,  (\ref{yu-5-2-1-b}), (\ref{yu-5-5-1-b}) and (\ref{yu-5-4-5-b}),
    we can find  $\mu_1''>0$ and  $M''_{1}>0$ so that
    each solution $(y_2(\cdot),y_1(\cdot),z_1(\cdot))^\top$ to  (\ref{yu-5-5-5-b})
     satisfies
\begin{eqnarray*}\label{yu-5-5-6-b}
\|( y_{2}(t), y_{1}(t),z_{1}(t))^{\top}\|_{\mathbb{R}^n}
\leq M''_{1}e^{-\mu_1'' t}\|( y_{2}(0), y_{1}(0),z_{1}(0))^{\top}\|_{\mathbb{R}^n}\;\;\mbox{for each}\;\;t\in\mathbb{R}^+.
\end{eqnarray*}
Third,   $(y_2(\cdot), y_1(\cdot),z_1(\cdot))^\top$ solves \eqref{yu-5-5-5-b} if and only if
  $x(\cdot):=(L_1L_{2})^{-1}(y_2(\cdot), y_1(\cdot),z_1(\cdot))^\top$ solves
   \eqref{yu-10-24-4} where $\Lambda_{\hbar}=\Lambda_{2}$ and $\mathcal{F}=\{F_k\}_{k=1}^2$ with
   $F_1$ and $F_2$
    given by \eqref{3.402019614}.

   Finally, the last two facts above lead to the $2$-stabilization of $(A,\{B_{k}\}_{k=1}^{2},\Lambda_{2})$.

  \vskip 5pt
  \noindent {\it Part 3:  We prove Proposition \ref {yu-proposition-5-13-2} for the case that $\hbar\geq 3$}

Arbitrarily fix $\hbar\geq 3$ and $\Lambda_{\hbar}\in \mathfrak{I}_{\hbar}\cap \mathfrak{L}_{A,\mathcal{B},\hbar}$.
 Let $\Lambda:=\{t_{j\hbar}\}_{j\in\mathbb{N}}\in\mathfrak{I}_1$. Denote
 $[C,D]:=(D,CD,\cdots,C^{p-1}D)$ for any $C\in\mathbb{R}^{p\times p}$ and $D\in\mathbb{R}^{p\times q}$ (with
 $p,q\in\mathbb{N}^+$).
 We start from $(A,B_\hbar)$. When $[A,B_\hbar]$ is full of rank, we turn to
\eqref{3.156-18-1} and use the same way there to get the desired stabilization. When $[A,B_\hbar]$ is not full of rank, we use Lemma \ref{yu-lemma-5-10-1} to get a decomposition for the pair $(A,B_\hbar)$
with an uncontrollable part $A_{\hbar,3}$ (see Step 1 in Part 2). Next, if $\sigma(A_{\hbar,3})\cap \mathbb{C}^+=\emptyset$, then we turn to Step 2 in Part 2 to get the desired
stabilization. If $\sigma(A_{\hbar,3})\cap \mathbb{C}^+\neq\emptyset$, we take
$B_{\hbar-1}$ into the consideration. Similar to \eqref{yu-5-4-b-1-b}, we divide $B_{\hbar-1}$ into two parts:
$\widehat{B}_{\hbar}$ and ${B}_{\hbar-1,\hbar}$. When $[A_{\hbar,3},{B}_{\hbar-1,\hbar}]$ is full of rank, we turn to Step 3 in Part 2 to get the desired  stabilization. When $[A_{\hbar,3},{B}_{\hbar-1,\hbar}]$ is not full of rank,
we turn to another decomposition for $(A_{\hbar,3},{B}_{\hbar-1,\hbar})$ with an uncontrollable part $A_{\hbar-1,3}$ (see Sub-step 4.1 in Part 2).
\par
We continue the above process. After $\hbar$ steps, if we are not done, then, based on
 $(A,B_{\hbar})$, $(A,B_{\hbar},B_{\hbar-1}), \ldots, (A,B_{\hbar},B_{\hbar-1},\cdots,B_1)$, we can get uncontrollable parts: $A_{\hbar,3},
A_{\hbar-1,3},\dots, A_{1,3}$ with $\sigma(A_{k,3})\cap \mathbb{C}^+\neq\emptyset$
for all $\hbar\geq k\geq 2$. Then we can use a very similar way as that used in  Sub-step 4.2 in Part 2, to get
$\sigma(A_{1,3})\cap \mathbb{C}^+\neq\emptyset$. From this, we can turn to Sub-step 4.3 in Part 2
to get the desired stabilization.

\vskip 5pt
In summary, we  end the proof of Proposition \ref{yu-proposition-5-13-2}. \end{proof}

\subsection {Proofs of Theorems \ref{yu-theorem-3-14-1}-\ref{yu-theorem-5-13-1}}

\begin{proof}[Proof of Theorem \ref{yu-theorem-3-14-1}]

From Proposition \ref{yu-proposition-5-13-2}, $(iii)\Rightarrow (i)$ follows at once. It is trivial that
    $(ii)\Rightarrow (iii)$. We now show  $(i)\Rightarrow (ii)$.

    Recall \eqref{1.162019611} for $\mathcal{B}$ which is in $\mathbb{R}^{n\times(m\hbar)}$.
        Without loss of generality, we  assume  $\mathcal{R}[A,\mathcal{B}]=\hat n<n$. (For otherwise, $(ii)$ of Theorem \ref{yu-theorem-3-14-1} follows from \cite[Lemma 3.3.7]{b22} at once.)
        Thus, by Lemma \ref{yu-lemma-5-10-1},
     there is an invertible matrix $L\in\mathbb{R}^{n\times n}$ such that
\begin{eqnarray}\label{3.46617wan}
    LAL^{-1}=\left(
               \begin{array}{cc}
                 A_1 & A_2 \\
                 0 & A_3 \\
               \end{array}
             \right),
    \;\;L\mathcal{B}=\left(
                       \begin{array}{cccc}
                         LB_1  & \cdots & LB_{\hbar} \\
                       \end{array}
                     \right)=\left(
                               \begin{array}{c}
                                 \widehat{\mathcal{B}} \\
                                 0 \\
                               \end{array}
                             \right),\; \mathcal{R}[A_1,\widehat{\mathcal{B}}]=\hat{n},
\end{eqnarray}
    where $A_1\in\mathbb{R}^{\hat{n}\times \hat{n}}$, $A_2\in\mathbb{R}^{\hat{n}\times (n-\hat{n})}$, $A_3\in\mathbb{R}^{(n-\hat{n})\times (n-\hat{n})}$ and $\widehat{\mathcal{B}}\in\mathbb{R}^{\hat{n}\times (m\hbar)}$.

    By the first two equalities in \eqref{3.46617wan} and by $(i)$ of Theorem \ref{yu-theorem-3-14-1}, we have
     $\sigma(A_3)\cap \mathbb{C}^+=\emptyset$ which implies
    \begin{eqnarray}\label{3.47617wan}
    \mbox{Rank}\;(\lambda \mathbb{I}_{n-\hat n}-A_3)=n-\hat n\;\;\mbox{for all}\;\;\lambda\in \mathbb{C}^+.
    \end{eqnarray}
    Meanwhile, by the last equality in \eqref{3.46617wan}, we can use \cite[Lemma 3.3.7]{b22} to get
    \begin{eqnarray}\label{3.48618wan}
     \mbox{Rank}\;(\lambda \mathbb{I}_{\hat n}-A_1, \widehat{\mathcal{B}})=\hat n \;\;\mbox{for all}\;\;\lambda\in \mathbb{C}.
    \end{eqnarray}

    Finally, by the first two equalities in \eqref{3.46617wan} and by \eqref{3.47617wan}
    and \eqref{3.48618wan}, we have
          \begin{eqnarray*}
     n\geq\mbox{Rank}\;(\lambda\mathbb{I}_n-A,\mathcal{B})\geq \mbox{Rank}\;(\lambda\mathbb{I}_{\hat{n}}-A_1,\widehat{\mathcal{B}})
    +\mbox{Rank}\;(\lambda\mathbb{I}_{n-\hat{n}}-A_3)=n\; \mbox{for all}\; \lambda\in\mathbb{C}^+,
     \end{eqnarray*}
     which leads to $(ii)$ of Theorem \ref{yu-theorem-3-14-1}.

\vskip 5pt
    In summary, we end the proof of Theorem \ref{yu-theorem-3-14-1}.
\end{proof}

\begin{proof}[Theorem \ref{yu-theorem-5-13-1}]
 By Theorem \ref{yu-theorem-3-14-1} and  Proposition \ref{yu-proposition-5-13-2}, Theorem \ref{yu-theorem-5-13-1} follows at once. This ends the proof of Theorem \ref{yu-theorem-5-13-1}.
\end{proof}

\section{Conclusions and perspectives}

Inspired by phenomena of multi-person cooperations, we set up a periodic impulse control system \eqref{yu-10-24-1}. Then we studied systematically the stabilization for this system:
 First, we obtained
 several necessary and sufficient conditions on  the stabilization of the system $(A, \{B_k\}_{k=1}^\hbar, \Lambda_{\hbar})$, and then gave a way to build up feedback laws.  Second, we got several  necessary and sufficient conditions on
the stabilization for the pair $(A, \{B_k\}_{k=1}^\hbar)$ and then provided locations where impulse instants should stay.

 In the studies of $(A, \{B_k\}_{k=1}^\hbar, \Lambda_{\hbar})$, the main ideas are originally from  the classical
LQ theory. But we modified the cost functional and derived a discrete dynamic programming principle which leads to the variant of Riccati's equation    \eqref{yu-11-26-1}. Both the discrete dynamic programming principle
and the variant of Riccati's equation differ from the classical ones.
 In the studies of $(A, \{B_k\}_{k=1}^\hbar)$, our method is based on the repeated use of Kalman controllable decomposition and a result in \cite{b2}.

Several open issues are given in order:

\begin{itemize}
\item Extensions of our main results to some infinite-dimensional systems.

\item Applications of our main results to non-linear systems.

\item  The relationship between the feedback law designed by usual LQ theory for the
control system (\ref{421usualcontrolsystem}) and our feedback law (\ref{yu-12-4-1-b}) (with
 $\hbar=1$, $B_1=B$, $\Lambda_1=\{i\tau\}_{j\in\mathbb{N}}$ $(\tau>0)$)  designed
by discrete LQ problem.
\end{itemize}

\end{document}